\documentclass[reqno]{amsart}  

\usepackage[latin2]{inputenc}

\usepackage{color}              
\usepackage[usenames,dvipsnames]{xcolor}
\usepackage{graphicx}
\usepackage[]{amsmath}
\usepackage{amsthm}
\usepackage{amssymb}
\usepackage{bbm}
\usepackage{comment}
\usepackage{microtype}
\usepackage[textsize=small]{todonotes}
\usepackage[colorlinks,citecolor=blue,urlcolor=blue,linkcolor=blue,pagecolor=blue,linktocpage=true,backref=true]{hyperref}

\usepackage{pdfsync}
\usepackage{tikz}
\usepackage{pgffor}

\usepackage{enumerate}
\newenvironment{enumeratei}{\begin{enumerate}[\upshape (i)]}{\end{enumerate}}
\newenvironment{enumeratea}{\begin{enumerate}[\upshape (a)]}{\end{enumerate}}

\numberwithin{equation}{section}

\newtheorem{theorem}{Theorem}[section]
\newtheorem{lemma}[theorem]{Lemma}

\newtheorem{corollary}[theorem]{Corollary}

\newtheorem{claim}[theorem]{Claim}

\newtheorem{remark}[theorem]{Remark}

\def \ind     {1\!\!1}

\newcommand{\abs}[1]{\left|{#1}\right|}

\newcommand{\Prob}[1]    {\ensuremath{\mathbb{P}\left(#1\right)}}
\newcommand{\Expect}[1]  {\ensuremath{\mathbb{E}\left[#1\right]}}
\newcommand{\Var}[1]  {\ensuremath{\mathbb{D}^2\left[#1\right]}}

\newcommand*{\Ev}{{\mathbb{ E}}}
\newcommand*{\Pv}{{\mathbb{ P}}}

\def \PP    {\mathbf{P}}
\def \toinp    {\buildrel {\PP}\over{\longrightarrow}}
\def \toindis  {\buildrel {d}\over{\longrightarrow}}

\def \Hop     {\mathrm{Hop}}
\def \Flood     {\mathrm{Flood}}
\def \Diam     {\mathrm{Diam}}
\definecolor{MyDarkBlue}{rgb}{0,0.08,0.50}
\definecolor{BrickRed}{rgb}{0.65,0.08,0}

\hypersetup{
colorlinks=true,       
    linkcolor=MyDarkBlue,          
    citecolor=BrickRed,        
    filecolor=red,      
    urlcolor=cyan           
}

\newcommand{\KJ}[1]{\todo[color=SkyBlue,inline]{Juli: #1}}

\newcommand{\N}{\mathbb{N}}

\newcommand{\CA}{\mathcal {A}}

\newcommand{\CG}{\mathcal {G}}

\newcommand{\CL}{\mathcal {L}}

\newcommand{\CN}{\mathcal {N}}

\newcommand{\dist}{{\rm dist}}

\newcommand{\diam}{{\rm diam}}

\newcommand*{\ve}{\varepsilon}

\newcommand*{\al}{\alpha}

\newcommand*{\RAND}{\mathrm{RAN}_d}
\newcommand*{\EAND}{\mathrm{EAN}_d}
\newcommand*{\RAN}{\mathrm{RAN}}
\newcommand*{\EAN}{\mathrm{EAN}}

\newcommand*{\be}{\begin{equation}}
\newcommand*{\ee}{\end{equation}}
\newcommand*{\ba}{\begin{aligned}}
\newcommand*{\ea}{\end{aligned}}
\newcommand*{\barr}{\begin{array}{c}}
\newcommand*{\earr}{\end{array}}

\begin{document}

		\title[Degrees and distances in RANs and EANs]{Degrees and distances in random and evolving Apollonian networks}


		\author[Kolossv\'ary]{Istv\'an Kolossv\'ary$^{1,3}$}
		\address{$^1$Budapest University of Technology and Economics,
		Inter-University Centre for Telecommunications and Informatics 4028 Debrecen, Kassai \'ut 26.}
		\author[Komj\'athy]{J\'ulia Komj\'athy$^{2,4}$}
		\address{$^2$Department of Mathematics and
		    Computer Science, Eindhoven University of Technology, P.O.\ Box 513,
		    5600 MB Eindhoven, The Netherlands.}

		\author[V\'ag\'o]{Lajos V\'ag\'o$^{1,3}$}
\thanks{$^3$This publication was partially supported by the T\'AMOP-4.2.2.C-11/1/KONV-2012-0001**project.The project has been supported by the European Union,
co-financed by the European Social Fund. This research was supported in part by the grant KTIA-OTKA \# CNK 77778, funded by the Hungarian National Development Agency (NF\"U) from a source provided by KTIA.\\
$^4$ The work of JK is supported in part by The Netherlands Organisation for
Scientific Research (NWO)}
		\email{istvanko@math.bme.hu, j.komjathy@tue.nl, vagolala@math.bme.hu}

	\begin{abstract}
		This paper studies Random and Evolving Apollonian networks ($\RAN$s and $\EAN$s), in $d$ dimension for any $d\ge 2$, i.e. dynamically evolving random $d$ dimensional simplices looked as graphs inside an initial $d$-dimensional simplex.  We determine the limiting degree distribution in $\RAN$s and show that it follows a power law tail with exponent $\tau=(2d-1)/(d-1)$. We further show that the degree distribution in $\EAN$s converges to the same degree distribution if the simplex-occupation parameter in the $n$th step of the dynamics is $q_n\to 0$ and $\sum_{n=0}^\infty q_n =\infty$. This result gives a rigorous proof for the conjecture of Zhang et al \cite{Zhang2006EAN} that $\EAN$s tend to show similar behavior as $\RAN$s once the occupation parameter $q\to 0$.
		We also determine the asymptotic behavior of shortest paths in $\RAN$s and $\EAN$s for arbitrary $d$ dimensions. For $\RAN$s we show that the shortest path between two uniformly chosen vertices (typical distance), the flooding time of a uniformly picked vertex and the diameter of the graph after $n$ steps all scale as constant times $\log n$. We determine the constants for all three cases and prove a central limit theorem for the typical distances. We prove a similar CLT for typical distances in $\EAN$s.	\end{abstract}

	\maketitle

\section{Introduction}\label{sec::intro}

The construction of deterministic and random Apollonian networks originates from the problem of Apollonian circle packing: starting with three mutually tangent circles, we inscribe in the interstice formed by the three initial circles the unique circle that is tangent to all of them: this fourth circle is known as the inner Soddy-circle. Iteratively, for each new interstice its inner Soddy-circle is drawn. After infinite steps the result is an Apollonian gasket \cite{Boyd1982, Graham20031}.

An Apollonian network (AN) is the resulting graph if we place a node in the center of each circle and connect two nodes if and only if the corresponding circles are tangent. This model was introduced independently by Andrade et al. \cite{Andrade:2005:detAN} and Doye and Massen \cite{Doye2005} as a model for networks arising in real-life such as the network of internet cables or links, collaboration network or protein interaction networks. Apollonian networks serve a good model for these networks since they share many similar features that most of them have: a power-law degree distribution, a high clustering coefficient and small distances, usually referred to as the small-world property. Moreover, by construction, Apollonian networks also show high hierarchical structure: a property that is very common in e.g.\ social networks.

It is straightforward to generalize Apollonian packings to arbitrary $d$ dimensions with mutually tangent $d$ dimensional hyperspheres. Analogously, if each $d$-hypersphere corresponds to a vertex and vertices are connected by an edge if the corresponding $d$-hyperspheres are tangent, then we obtain a $d$-dimensional AN (see \cite{Zhang:2008:APLforAN, Zhang2006:ddimAN}).

The network arising by this construction is deterministic. Zhou et al. \cite{Zhou:2005:RAN} proposed to randomise the dynamics of the model such that in one step only one interstice is picked uniformly at random and filled with a new circle. This construction in $d$ dimensions yields a $d$ dimensional \emph{random Apollonian network} ($\RAN$) \cite{Zhang:2006:ddimRAN}. Using heuristic and rigorous arguments the results in \cite{Albenque:2008:hopinRAN, Cooper2013diamRAN, Darrasse:2007:RANstruct, Abbas2013diamRAN, Tsourakis:2012:wrongRANdiam, Zhang:2006:ddimRAN, Zhou:2005:RAN}  show that RANs have the above mentioned main features of real-life networks.

A different random version of the original Apollonian network was introduced by Zhang et al. \cite{Zhang2006EAN}, called \emph{Evolutionary Apollonian networks} ($\EAN$) where in every step \textit{every} interstice is picked and filled independently of each other with probability $q$. If an interstice is not filled in a given step, it can be filled in the next step again. We call $q$ the \emph{occupation parameter}. For $q=1$ we get back the deterministic AN model. It is conjectured in \cite{Zhang2006EAN} that if $q\to 0$ then $\EAN$s show similar features to $\RAN$s. To answer this issue, we investigate how  the network looks like when $q:=q_n$ depends on the time and tends to $0$ as the size of the network grows. In this setting, the interesting question is to determine the correct rate for $q_n$ that achieves that $\EAN$ shows similar behavior as $\RAN$.

\subsection*{Our contribution}
In this paper we give a rigorous proof for the power law degree distribution of $\RAN$s in any dimension $d$. Further, we show that if we let the occupation parameter of $\EAN$ to depend on $n$ and choose such that $q_n\to 0$ and $\sum_{n=0}^\infty q_n =\infty$, then the limiting degree distribution of $\EAN$s is a power law with the same exponent as for $\RAN$s. We also determine the clustering coefficient of $\RAN$s and $\EAN$s.
Further, we investigate the asymptotic behavior of shortest paths in random Apollonian networks in arbitrary $d$ dimensions with rigourous methods. We show that the shortest path between two uniformly chosen vertices (typical distance), the flooding time of a uniformly picked vertex and the diameter of the graph after $n$ steps all scale as  constant times $\log n$. We determine the constants for all three cases and prove a central limit theorem for the typical distances.
We introduce a different approach to describe the structure of the graph that enables us to unfold the hierarchical structure of $\RAN$s and $\EAN$s and is based on coding used in fractal-related methods.

Now we give the precise definition of the two models.

\subsubsection*{Random Apollonian networks}\label{subsec::RAN}

A random Apollonian network $\mathrm{RAN}_d(n)$ in $d$ dimensions can be constructed as follows. Initially at step $n=0$ start from the $d$-dimensional simplex with an additional vertex in the interior connected to all of the vertices of the simplex. Thus there are initially $d+1$ $d$-simplices, that we call  \emph{active cliques}. For $n\geq1$, pick
an active clique $C_n$ of $\mathrm{RAN}_d(n-1)$ \textit{uniformly at random}, insert a node $v_n$ in the interior of $C_n$ and connect $v_n$ with all the vertices of $C_n$. The newly added vertex $v_n$ forms new cliques with each possible choice of $d$ vertices of $C_n$. $C_n$ becomes inactive and these newly formed $d$-simplices become active. The resulting graph is $\mathrm {RAN}_d(n)$. At each step $n$ a $\mathrm{RAN}_d(n)$ has $n+d+2$ vertices and $nd + d+1$ active cliques.

There is a natural representation of RANs as evolving triangulations in two dimensions: take a planar embedding of the complete graph on four vertices as in Figure \ref{fig:ran1} and in each step pick a face of the graph uniformly at random, insert a node and connect it with the vertices of the chosen triangle. The result is a maximal planar graph. Hence, a $(\mathrm{RAN}_2(n))_{n\in \N}$ is equivalent to an increasing family of triangulations by successive addition of faces in the plane, called \emph{stack-triangulations}. Stack-triangultions were investigated in \cite{Albenque:2008:hopinRAN}  where the authors also considered typical properties under different weighted measures, (e.g.\ uniformly picked ones having $n$ faces).
Under a certain measure stack-triangulations with $n$ faces are an equivalent formulation of $\RAN_2(n)$, see \cite{Albenque:2008:hopinRAN} and references therein.

\begin{figure}[!h]
\begin{tikzpicture}
[scale=0.44, thick,
 initial/.style={circle, shade, ball color=black!40, scale=0.8}
]

\path[draw=black] (0,0) node[initial] (A) {} -- (5,3.5) node[initial] (O) {} -- (4,7) node[initial] (B) {} -- cycle;
\node[initial] at (8,0) (C) {};
\foreach \x in {A,B,O}
  \path[draw=black] (C) -- (\x);


\begin{scope}
[ xshift=9.5cm, halott/.style={circle, shade, ball color=black!30, scale=0.6} ]

\path[draw=black] (0,0) node[initial] (A) {} -- (5,3.5) node[initial] (O) {} -- (3.5,5.5) node[halott] (1) {} -- cycle;
\node[initial] at (4,7) (B) {};
\foreach \x in {A,O}
  \path[draw=black] (B) -- (\x);
\node[initial] at (8,0) (C) {};
\foreach \x in {A,B,O}
  \path[draw=black] (C) -- (\x);
\node[halott] at (3.5,5.5) (1) {};
\foreach \x in {A,B,O}
  \path[draw=black] (1) -- (\x);
\node[halott] at (2.2,2.8) (13) {};
\foreach \x in {A,O,1}
  \path[draw=black] (13) -- (\x);
\end{scope}

\begin{scope}
[ xshift=19cm, halott/.style={circle, shade, ball color=black!30, scale=0.6}]

\path[draw=black] (0,0) node[initial] (A) {} -- (8,0) node[initial] (C) {} -- (4,7) node[initial] (B) {} -- (A);
\node[initial] at (5,3.5) (O) {};
\foreach \x in {A,B,C}
  \path[draw=black] (O) -- (\x);
\node[halott] at (4.6,1.5) (3312) {};
\node[halott] at (3.7,3.9) (132) {};

\foreach \x/\y/\z in {3.5/5.5/1,4.9/2.6/3,2.2/2.8/13,3.8/2.4/31,6.6/0.7/33,2.8/1/331}
  \node[halott] at (\x,\y) (\z) {};
\path[draw=black] (O) -- (1) -- (13) -- (132);
\path[draw=black] (O) -- (3) -- (33) -- (331) -- (3312);
\path[draw=black] (3) -- (31);

\foreach \x in {1,13,31,3,331,33}
  \path[draw=black] (A) -- (\x);
\foreach \x in {3,33}
  \path[draw=black] (C) -- (\x);
\foreach \x in {13,132,31}
  \path[draw=black] (O) -- (\x);
\path[draw=black] (B) -- (1);
\foreach \x in {3,33}
  \path[draw=black] (3312) -- (\x);
\path[draw=black] (1) -- (132);
\path[draw=black] (3) -- (331);
\end{scope}

\end{tikzpicture}\label{fig:ran1}
\caption{A $\mathrm{RAN}_2(n)$ after $n=0,\, 2,\, 8$ steps}
\end{figure}
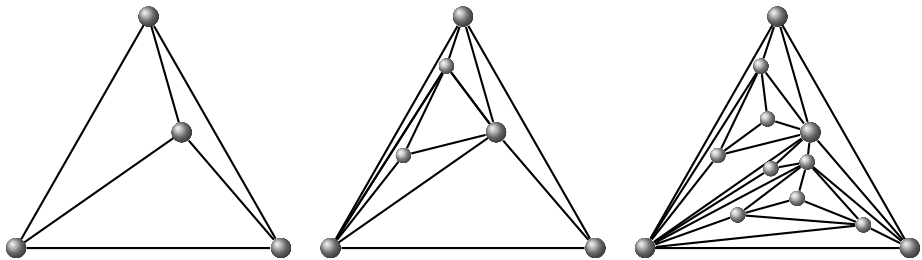

\subsubsection*{Evolutionary Apollonian networks}\label{subsec::EAN}

Given a sequence of occupation parameters $\{q_n\}_{n=1}^{\infty}, 0\le q_n\leq1$, an evolutionary Apollonian network $\mathrm{EAN}_d(n,\{q_n\})=\mathrm{EAN}_d(n)$ in $d$ dimensions can be constructed iteratively as follows. The initial graph is the same as for a $\mathrm{RAN}_d(0)$ and there are $d+1$ active $d$-simplices. For $n\geq1$, pick
each active clique of $\mathrm{EAN}_d(n-1)$ independently of each other with probability $q_n$. The set of chosen cliques $\mathcal C_n$ becomes inactive (the non-picked active cliques stay active) and for every clique $C\in\mathcal{C}_n$ we place a new node $v_n(C)$ in the interior of $C$ that we connect to all vertices of $C$. This new node $v_n(C)$ together with all possible choices of $d$ vertices from $C$ forms $d+1$ new cliques: these cliques are added to the set of active cliques for every $C\in \mathcal C_n$. The resulting graph is $\mathrm{EAN}_d(n)$. The case $q_n\equiv q$ was studied in \cite{Zhang2006EAN} where it was further suggested that for $q\to 0$ the graph is similar to a $\mathrm{RAN}_d(n)$. We prove their conjecture by showing that $\EAN$s obey the same power law exponent as $\RAN$s if $q_n\to 0$ and $\sum_{n=0}^\infty q_n =\infty$.

\subsubsection*{Structure of the paper}
In Section \ref{sec::MainRes} we state our main results and discuss their relation to other results in the area. Section \ref{sec::strucureofRAN} contains the most important observations about the structure of $\RAN$s: we work out an approach of coding the vertices of the graph that enables us to compare the structure of the $\RAN$ to a branching process and further, the distance between any two vertices in the graph is given entirely by the coding of these vertices. We also give a short sketch of proofs related to distances in this section. Then we prove rigorously the distance-related theorems in Section \ref{sec::ProofDist}.
Finally in Section \ref{sec::ProofDegDistr} we prove the results concerning the degree distributions.
\section{Main results}\label{sec::MainRes}

In this section we state our main results and compare it to related literature. First we start with theorems about distances in $\RAN$s and $\EAN$s, then state our theorems about the limiting degree distributions and clustering coefficient in $\RAN$s and $\EAN$s, respectively, and finally we compare our results to related literature.

\subsection{Distances in $\RAN$s and $\EAN$s}\label{subsec::dist}
In this subsection we state our results about distances in $\RAN$s and $\EAN$s.
First we define the three main quantities of our investigation:
fix $n$ and pick two vertices $u$ and $v$ uniformly at random from $\RAND(n)$. Denote by $\Hop_d(n,u,v)$ the hopcount between the vertices $u$ and $v$, i.e. the number of edges on (one of) the shortest paths between $u$ and $v$. The flooding time $\Flood_d(n,u)$ is the maximal hopcount from $u$, while the diameter $\Diam_d(n)$ is the maximal flooding time, formally
\begin{equation*}
\Flood_d(n,u)=\max_v \Hop_d(n,u,v) \;\text{ and }\; \Diam_d(n)=\max_{u,v}\Hop_d(n,u,v).
\end{equation*}
Whenever possiblem $d$, $u$ and $v$ are suppressed from the notation. We introduce some important notation first.
Let
\begin{equation}\label{def:y}
Y_d:= \sum_{i=1}^{d+1} X_i
\end{equation}
where $(X_i)_{i=1}^{d+1}$ is a collection of independent geometrically distributed random variables with success probability
$\frac{i}{d+1}$.
We further introduce
\begin{equation}\label{def:mud}
\mu_d:=\Ev[Y_d]=(d+1) H(d+1), \quad \sigma_d^2:=\Var{Y_d},
\end{equation}
where $H(d)=\sum_{i=1}^{d}1/i.$
The Large Deviation rate function of $Y_d$ is given by
\be\label{def:rate-function} I_d(x):= \sup_{\lambda\in\mathbb{R}} \left\{ \lambda x - \log \left( \Ev\left[\mathrm{e}^{\lambda Y_d}\right]\right)\right\}. \ee
The next theorem describes the asymptotic behavior of the typical distances in $\RAN_d(n)$.
\begin{theorem}[Typical distances in $\RAN$s]\label{thm::MainRes}
The hopcount between the vertices corresponding to two uniformly chosen active cliques in a $\mathrm{RAN}_d(n)$ satisfies a central limit theorem (CLT) of the form
\begin{equation}\label{eqn::mainHop}
\frac{\Hop_d(n)- \frac{2}{\mu_d}\frac{d+1}{d}\log n} {\sqrt{2 \frac{\sigma^2_d+\mu_d}{\mu^3_d }\frac{d+1}{d}\log n}} \toindis Z,
\end{equation}
where $\mu_d,\sigma_d^2$ as in \eqref{def:mud} and $Z$ is a standard normal random variable.

Further, the same CLT is satisfied for the distance between two vertices that are picked independently with the size-biased probabilities given by
\be\label{eq:biased}\Pv(v\mbox{ is picked }| D_v(n)=k ) = \frac{(d-1)k-d^2+d+2}{d n +d+1}.\ee

\end{theorem}
The next theorem describes the asymptotic behaviour of the flooding time and the diameter:
\begin{theorem}[Diameter and flooding time in $\RAN$s]\label{thm::Diam}
Define $\widetilde c_d$ as the unique solution with $\tilde c_d>\frac{d+1}{d}$ to the equation
\begin{equation*}
f_d(c):=c-\frac{d+1}{d}- c\log\Big(\frac{d}{d+1} c\Big)=-1.
\end{equation*}
Then as $n\to \infty$ with high probability
\be\ba \label{eqn::mainFlood&Diam}
\frac{\Diam_d(n)}{\log n} &\toinp 2\tilde\alpha\tilde\beta\frac{\tilde c_d}{\mu_d},\\
\frac{\Flood_d(n)}{\log n} &\toinp \frac{1}{\mu_d}\left(\frac{d+1}{d}+\tilde\alpha\tilde\beta\tilde c_d\right).
\ea
\ee
where $(\tilde\alpha,\tilde\beta)\in (0,1]\times[1,\frac{\mu_d}{d+1}]$ is the solution of the maximization problem with the following constraint:
\begin{equation}\label{eqn::defalpha*beta}
\max \{ \alpha \beta :  1+f(\alpha \tilde c_d)-\alpha\beta \frac{\tilde c_d}{\mu_d}I_d\left(\frac{\mu_d}{\beta}\right)=0 \},
\end{equation}
where $\mu_d,\sigma^2_d$ as in \eqref{def:mud} and $I_d(x)$ as in \eqref{def:rate-function}.
\end{theorem}
\begin{remark}\normalfont
Observe that the set of $(\alpha,\beta)$ pairs that satisfy \eqref{eqn::defalpha*beta} is non-empty since for $\alpha=\beta=1$ by definition $f(\tilde c_d)=-1$ and $I_d(\mu_d)=0$.
\end{remark}
Finally, the next theorem describes the asymptotic behavior of the typical distances in $\EAN_d(n)$.
\begin{theorem}[Typical distances in $\EAN$s]\label{thm::MainResEAN}
Suppose the sequence of occupation parameters $\{q_n\}$ satisfies $\sum_{n\in \N}q_n=\infty$ and $\sum_{n\in \N}q_n(1-q_n)=\infty$.
Then the hopcount between vertices corresponding to two uniformly chosen active cliques in a $\mathrm{EAN}_d(n)$ satisfies a central limit theorem of the form
\begin{equation}\label{eqn::mainHopEAN}
\frac{\Hop_d(n)- \frac{2}{\mu_d}\sum\limits_{i=1}^n q_i} {\sqrt{2 \frac{\sigma^2_d+\mu_d}{\mu^3_d }\sum\limits_{i=1}^n q_i(1-q_i)}} \toindis Z,
\end{equation}
where $\mu_d,\sigma_d^2$ as in \eqref{def:mud} and $Z$ is a standard normal random variable.

Further, the same CLT is satisfied for the distance between two vertices that are picked independently with the size-biased probabilities given by
\be\label{eq:biasedEAN}\Pv(v\mbox{ is picked }| D_v(n)=k | |V(n)|) = \frac{(d-1)k-d^2+d+2}{d |V(n)| +d+1}.\ee

\end{theorem}
\begin{remark} \normalfont Note that in this theorem $q_n$ might or might not tend to $0$. The second criterion rules out the case that the $q_n\to 1$ and so the graph becomes essentially deterministic.
\end{remark}

\subsection{Degree distribution and clustering coefficient}\label{subsec::degreedistr}

First we describe the limiting distribution of $\RAND$.
Let us denote by $\widetilde N_k(n)$ and  $\widetilde p_k(n)$ the number and the empirical proportion of vertices with degree $k$ at time $n$ respectively, i.e.
\[
  \widetilde p_k(n):=\frac{\widetilde N_k(n)}{n+d+2}:=\frac{1}{n+d+2}\sum_{i=1}^{n+d+2}{\ind \{D_i(n)=k\} },
\]
where $D_i(n)$ stands for the degree of vertex $i$ after the $n$-th step.
Our first theorem describes that this empirical distribution tends to a proper distribution in the $\ell_\infty$-metric:
\begin{theorem}[Degree distribution for $\RAN$s]\label{739}
There exists a probability distribution $\{ p_k\}_{k=d+1}^{\infty}$ and a constant $c$ for which
\[
  \Pv \left( \max_k \left|  \widetilde p_k(n)-p_k\right| \geq c\sqrt{\frac{\log n}{n}}\right)=o(1).
\]
Further, $p_k$ follows a power law with exponent $(2d-1)/(d-1) \in (2,3]$,  more precisely
\be \label{eq:pkdef}
  p_k=\frac{d}{2d+1}\frac{\Gamma(k-d+\frac{2}{d-1})}{\Gamma(1+\frac{2}{d-1})}\: \frac{\Gamma(2+\frac{d+2}{d-1})}{\Gamma(k+1-d +\frac{d+2}{d-1})}=k^{- \frac{2d-1}{d-1}} (1+ o(1)),
\ee
where $\Gamma(x)$ is the Gamma function and we used the property that $\Gamma(t+a)/\Gamma(t)=t^a(1+o(1)).$
\end{theorem}

Our next theorem describes the degree distribution of the graph $\mathrm{EAN}_d(n,\{q_n\})$. Let us denote  the set of vertices after $n$ steps by $  V(n)$, and write $  N_k(n)$ and  $  p_k(n)$ for the number and the empirical proportion of vertices with degree $k$ at time $n$, respectively, i.e.\
\be\label{def:pk}
    p_k(n):= \frac{  N_k(n)}{  |V(n)|}=\frac1{|V(n)|} \sum_{i \in   V(n)}{\ind \{ D_i(n)=k\} }.
\ee
\begin{theorem}[Degree distribution for $\EAN$s]\label{k14}
Let $d\geq 2$ and $\{ q_n\}_{n=0}^\infty$ be probabilities such that
\[
  q_n\to 0, \quad \sum_{n=0}^\infty q_n=\infty.
\]
Then the degree distribution is tending to the same asymptotic degree distribution $\{   p_k\}_{k=d+1}^{\infty}$ as in the case of $\RAND$ \eqref{eq:pkdef}, and
\[
  \Pv \left( \max_k \left|   p_k(n)-  p_k\right| \geq c\sqrt{\frac{\log  |V(n)|}{|V(n)|}}\right)=o(1)
\]
for any $c>\sqrt{8}(d+1)^{3/2}$.
\end{theorem}
The proof of these theorems are given in Sections \ref{sec::proofDegreeRAN} and \ref{sec::proofDegreeEAN}.
Next we describe the clustering coefficient  of $\RAN$s and $\EAN$s. The clustering coefficient of a node is just the portion of the number of existing edges between its neighbors, compared to the number of all possible edges between those. Here we investigate the clustering coefficient of the whole graph, which is the average of clustering coefficients over the nodes. Since these are direct consequences of the formula for the limiting degree distributions $p_k$, we state them as corollaries. This corollary is similar to the result in \cite[Section 4.2.]{Zhang:2006:ddimRAN}, but now, that we have established the degree distribution, it has a rigorous proof.
This is based on the observation that the clustering coefficient of a vertex with degree $k$ is \emph{deterministic} and equals
\[
  \frac{d(2k-d-1)}{k(k-1)}\sim \frac{2d}k.
\]
The reason for this formula is the following: when the degree of a node $v$ increases by one by adding a new vertex $w$ in one of the active cliques $v$ is contained in, then the number of edges between the neighbors of $v$ increases exactly by $d$, since the newly added vertex $w$ is connected to the $d$ other vertices in the clique.
It was observed in simulations and heuristicly proved in \cite{Zhang:2006:ddimRAN}, that the average clustering coefficient of these networks are converging to a strictly positive constant. Our next corollary determines the exact value of these constants for the two models:
\begin{corollary}[Clustering coefficient]\label{cor::ClustCoeff}
The average clustering coefficient of $\mathrm{RAN}_d(n)$ converges to a strictly positive constant as $n\to\infty$, given by
\be\label{eq:clust} \begin{aligned}
  Cl_d&=\sum_{k=d+1}^\infty{\frac{d(2k-d-1)}{k(k-1)}p_k}\\
  &=\sum_{k=d+1}^\infty{\frac{d(2k-d-1)}{k(k-1)}\cdot\frac{d}{2d+1}\: \frac{\Gamma(k-d+\frac{2}{d-1})}{\Gamma(1+\frac{2}{d-1})}\: \frac{\Gamma(1+\frac{2d+1}{d-1})}{\Gamma(k-d+\frac{2d+1}{d-1})}}.
\end{aligned} \ee
Further, the clustering coefficient of $\mathrm{EAN}_d(n,\{q_n\})$ converges to the same value as in \eqref{eq:clust} if $q_n\to 0$ and $\sum_{n\in\N}q_n=\infty$.

\end{corollary}

\subsection{Related literature}\label{subsec::RelatedLitr}

Several results connected to the asymptotic degree distribution of Apollonian networks are known. It is not hard to see that if a vertex belongs to $k$ active cliques, then the chance of  getting a new edge is proportional to $k$: this argument shows that these models belong to the wide class of \emph{Preferential attachment models} \cite{BA,BR,BRST}. As a result, some of the classical methods can be adapted to this model.

Using heuristic arguments, Zhang and co-authors \cite{Zhang:2006:ddimRAN} obtained that the asymptotic degree exponent should be $\frac{2d-1}{d-1}\in (2,3]$, which is in good agreement with their simulations.
Parallel to writing this paper, we noticed that Frieze and Tsourakakis \cite{Tsourakis:2012:wrongRANdiam} very recently derived rigorously the exact asymptotic degree distribution of the two-dimensional $\RAN$.
Even though our work is parallel to theirs, the methods are similar: this is coming from the fact that both of the methods used there and in our work are an adaptation of standard methods given in \cite{BRST, Remco}. So, to avoid repetition we decided to only sketch some parts of the proof and include the part that does not follow from their work.

What is entirely new in our paper is that we study the $\EAN$ model rigorously. For the degree distribution of $\EAN$s only heuristic arguments were known before. Zhang, Rong and Zhou \cite{Zhang2006EAN} studied the graph series $\mathrm{EAN}_d(n,\{q_n\equiv q\})$. They derived the asymptotic degree exponent using heuristic arguments, and the result fits well with the simulations. They also suggested that as $q\to 0$ the model $\mathrm{EAN}_d(n,\{q_n\equiv q\})$ converges to $\mathrm{RAN}_d(n)$ in some sense. We confirm their claim by deriving the asymptotic degree distribution of $\mathrm{EAN}_d(n,\{q_n\})$ with $\{q_n\}$ such that $q_n\to 0$ and $\sum_{n=0}^\infty q_n =\infty$, obtaining the same degree distribution. So the idea of Zhang and co-authors can be made precise in this way.

The statements of Theorem \ref{thm::MainRes} are in agreement with previous results. In particular, in \cite{Zhang:2006:ddimRAN} the authors estimate the average path length, i.e. the hopcount averaged over all pairs of vertices, and they show that it scales logarithmically with the size of the network.

A refined, but still weaker claim is obtained by Albenque and Marckert \cite{Albenque:2008:hopinRAN} concerning the hopcount in two dimensions. They prove that
  \begin{equation*}
  \frac{\Hop(n)}{6/11 \log n} \toinp 1.
  \end{equation*}
The constant $6/11$ is the same as $2(d+1)/(d\mu_d)$ for $d=2$. They use the previously mentioned notion of stack triangulations to derive the result from a CLT similar to the one in Theorem \ref{thm::MainRes}, we show an alternative approach using weaker results. The CLT for distances in $\EAN$s is new.

Central limit theorems of the form \eqref{eqn::mainHop} for the hopcount have been proven with the addition of exponential or general edge weights for various other random graph models, known usually under the name \emph{first passage percolation}. Janson \cite{Janson_CompleteGraph} analysed distances in the complete graphs with i.i.d. exponential edge weights. In a series of papers Bhamidi, van der Hofstad and Hooghiemstra determine typical distances and prove CLT for the hopcount for the Erd\H{o}s-R\'enyi random graph \cite{RvdH_FPP}, the stochastic mean-field model \cite{RvdH_FPPonComplete}, the configuration model with finite variance degrees \cite{RvdH_FPPonCM} and quite recently for the configuration model \cite{RvdH_FPP_Generalweights} with arbitrary independent and identically distributed edge weights.  Inhomogeneous random graphs are handled by Bollob\'as, Janson and Riordan \cite{BB_IHRGM, KomKol2013FPPIHRG}. Note that in all these models the edges have random weights, while in $\RAN$s and $\EAN$s all edge weights are $1$. The reason for this similarity is hidden in the fact that all these models have an underlying branching process approximation, and the CLT valid for the branching process implies CLT for the hopcount on the graph. The diameter and flooding time of $\EAN$s remains a future project.

Further, there are some previous bounds known about the diameter of $\RAN$s: Frieze and Tsourakakis \cite{Tsourakis:2012:wrongRANdiam} establishes the upper bound $2\tilde c_2 \log n$ for $\RAN_2(n)$. They use a result of Broutin and Devroye \cite{Broutin:2006:LargeDev} that, combined with the branching process approximation of the structure of $\RAN$s we describe in this paper, actually implicitly gives the $2\tilde c_d\log n$ upper bound for all $d$.

Just recently and independently from our work other methods were used to determine the diameter. In \cite{Abbas2013diamRAN} Ebrahimzadeh et al. apply the result of \cite{Broutin:2006:LargeDev} in an elaborate way, while Cooper and Frieze in \cite{Cooper2013diamRAN} use a more analytical approach solving recurrence relations. We emphasize that the methods in \cite{Cooper2013diamRAN, Abbas2013diamRAN} and in the present paper are all qualitatively different. Numerical solution of the maximization problem \eqref{eqn::defalpha*beta} for $d=2$ yields the optimal $(\tilde\alpha,\tilde\beta)$ pair to be approximately $(0.8639,1.500)$. The corresponding constant for the diameter is $1.668$, which perfectly coincides with the one obtained in \cite{Cooper2013diamRAN} and \cite{Abbas2013diamRAN}. To the best of our knowledge no result has been proven for the flooding time.

\section{Structure of RANs and EANs}\label{sec::strucureofRAN}

\subsection{Tree-structure of RANs and EANs} The construction method of $\RAN$s and $\EAN$s enables us to describe a natural way to code the vertices and active cliques of the graph parallel to each other. Let $\Sigma_d:=\{1,2,\ldots,d+1\}$ be the symbols of the alphabet. We give the initial vertices of a $\mathrm{RAN}_d(0)$ the empty word except for the vertex in the middle of the initial $d$-simplex which gets the symbol $\mathbf{O}$ (root). Then, we label each initial active clique by a different symbol from $\Sigma_d$. In step $n=1$ we assign the newly added vertex $u$ the code $\mathbf{u}=i$ if the clique with label $i$ was chosen. Further, we label the $d+1$ newly formed $d$-simplices by $ij$ for all $j\in\Sigma_d$. Similarly for $n\geq2$, we assign the newly added vertex $v$ the \emph{label of the clique that becomes inactive} and the newly formed active cliques are given the labels $\mathbf{v}j$, $j=1,\ldots, d+1$. It is crucial to keep the labeling \emph{consistent} in a geometrical sense, we are going to describe how to do this in Lemma \ref{lemma::neighbors} below.

Call the length of a code $\abs{\mathbf{u}}$ the generation of the vertex $u$ and denote the deepest common ancestor of $u$ and $v$ by $\mathbf{u}\wedge \mathbf{v}$, thus $\abs{\mathbf{u}\wedge \mathbf{v}} = \min\{k: u_{k+1}\neq v_{k+1}\}$. For codes $\mathbf{u}=u_1\ldots u_n$ and $\mathbf{v}=v_1\ldots v_m$ denote the concatenation $u_1\ldots u_nv_1\ldots v_m$ by $\mathbf{u}\mathbf{v}$. Further, let $u_{(i)}$ denote the position of the last occurrence of the symbol $i\in\Sigma_d$ in $\mathbf{u}$. We introduce the cut-operators $T_i \mathbf{u}:= u_1\dots u_{(i)-1}$ and $P_i \mathbf{u}=u_{(i)}\dots u_n$ for all $i\in \Sigma_d$. The following lemma collects the most important combinatorial observations.

\begin{lemma}[Tree-like properties of the coding]\label{lemma::neighbors}
There exists a way of choosing the coding of the vertices in the following way:
\begin{enumeratea}
\item {\bf consistency:} The $d+1$ neighbors of a newly formed vertex $u$ with code $\mathbf u$ are $T_i \mathbf{u}$, $i\in\Sigma_d$.
\item The shortest path between any two vertices $\mathbf u$ and $\mathbf v$ goes through $\mathbf{u}\wedge \mathbf{v}$ or one of its ancestors.
\item For every $\mathbf{u}, \mathbf{v}, \mathbf{w} \,\, \Hop(uv,u)\leq\Hop(uvw,u)$. \item Assume $\abs{\mathbf{u}\wedge \mathbf{v}}<\min\{\abs{\mathbf{u}}, \abs{\mathbf{v}}\}$. Then $\Hop(u,v)\leq\Hop(uw,v)$.
\end{enumeratea}
\end{lemma}

Part (a) means that if we have a vertex with code $\mathbf u$, the first $d+1$ neighbours of the vertex can be determined by cutting off the last pieces of the code of $u$, up to the last occurence of a given character $i\in \Sigma_d$. This, in two dimensions means the following: suppose at step $n=0$ the `left', `right' and `bottom' triangles were given the symbols $1,2$ and $3$ respectively. Then later each new node $v$ with code $\mathbf v$ in the middle of  a triangle gives rise to the new `left', `right' and `bottom' triangles: to these we \emph{have to} assign the labels $\mathbf{v}1, \mathbf{v}2$ and $\mathbf{v}3$ respectively.
Part (b) says that if we have two vertices with code $\mathbf u$ and $\mathbf v$ in the tree, then the shortest path between them must intersect a path from the root to their deepest common ancestor  $\mathbf{u}\wedge \mathbf{v}$. Part (c) ensures that the shortest path between a vertex with code $\mathbf{u}$ and one of its descendants $\mathbf{uv}$ cannot go \emph{below} the vertex $\mathbf{uv}$ in the tree. In more detail, this means that even though there are non-tree (=shortcut) edges upwards in the tree along path towards the origins, these shortcut edges do not jump over each other: the shortcut edges from $\mathbf{uvw}$ (an arbitrary descendant of $\mathbf{uv}$) can only go above $\mathbf{uv}$ to vertices to which $\mathbf{uv}$ is also connected to. Finally, part (d) is a corollary of part (b) and (c): it means that if two vertices $u,v$ are not descendants of each other, then the shortest path between them does not go below them in the tree.

\begin{proof}

We prove part (a) by induction. We label the initial $d+1$ active cliques arbitrarily with $i\in \Sigma_d$. The hypotheses clearly holds in this case. Suppose now we already have an active clique with code $\mathbf u$, and by induction, we can assume that $\mathbf u$ is associated with the clique formed by vertices $(T_1 \mathbf u, T_2 \mathbf u \dots, T_{d+1} \mathbf u )$. When this clique becomes inactive, the vertex $\mathbf{u}$ is added to the graph and the new active cliques are
\begin{equation}\label{eq:activecliques}
(T_1 \mathbf u, \dots, T_{k-1} \mathbf u, T_{k+1} \mathbf u, \dots, T_{d+1} \mathbf u, \mathbf u) \mbox{ for all } k\in \Sigma_d.
\end{equation}
 Let us denote this clique by $\mathbf uk$.
We claim that by this choice, the property (a) is maintained, that is, if the vertex  $\mathbf uk$ will ever be added to the graph, then its neighbours going to be $T_i (\mathbf u k), i\in \Sigma_d$.
By construction, the neighbours of $\mathbf uk$ are exactly the ones in \eqref{eq:activecliques}, and clearly we have $T_k(\mathbf u k ) = \mathbf u$, and $T_i(\mathbf uk) = T_i(\mathbf u)$ for $i\neq k$, so we can write
\[ \begin{aligned}&(T_1 \mathbf u, \dots, T_{k-1} \mathbf u, T_{k+1} \mathbf u, \dots, T_{d+1} \mathbf u, \mathbf u) \\
&\qquad= (T_1 (\mathbf uk), \dots, T_{k-1} (\mathbf uk), T_{k+1} (\mathbf uk), \dots, T_{d+1} (\mathbf uk), T_k(\mathbf uk)). \end{aligned}\]
This proves (a).

(b) Note that by part (a), every vertex is connected to $d+1$ vertices which code length is shorter than $|\mathbf u|$, and all these vertices are descendants of each other, i.e. they are in the path from $\mathbf u$ to the root. The other vertices $\mathbf u$ is connected to are its descendants, i.e.  of the form $\mathbf {uw}$ for some $\mathbf w$. Hence if we want to build a path from vertex $u$ to $v$, we must go up in the tree to at least $\mathbf u\wedge \mathbf v$.

For (c) it is enough to observe that the position of the last occurrence of a symbol in a code can not be earlier than in some prefix of the same code, i.e.
$|T_i (\mathbf{uvw})| > |T_i (\mathbf {uv})|$ for all $i$.

(d) follows from (b) and (c), since $\mathbf{u}\wedge\mathbf{v}=\mathbf{uw}\wedge\mathbf{v}$ and $\Hop(\mathbf u, \mathbf u\wedge \mathbf v)\leq\Hop(\mathbf{uw}, \mathbf u\wedge \mathbf v)$ and also for all ancestors of $\mathbf{u}\wedge\mathbf{v}$.
\end{proof}

The coding in turn gives a natural grouping of the edges. Edges of the initial graph are not given any name. An edge is called a \emph{forward edge} if its endpoints have codes of the form $\mathbf{u}$ and $\mathbf{u}j$ for $j\in\Sigma_d$. All other edges are called \emph{shortcut edges}. So at each step one new forward edge and $d$ shortcut edges are formed. Figure 2 below shows an example in two dimensions. Initially the symbols $1,2,3$ were assigned to the left, right, bottom triangles and continued with the coding the same way.

\begin{figure}[!h]
\begin{tikzpicture}
[scale=0.6,
 elo/.style={circle, shade, ball color=white, scale=0.6},
 halott/.style={circle, shade, ball color=black!40, scale=0.6},
 el/.style ={thick, double = black, double distance = 0.5pt}]

\path[el,draw=black] (0,0) node[halott] (A) {} -- (8,0) node[halott] (C) {} -- (4,7) node[halott] (B) {} -- (A);
\node[elo] at (5,3.5) (O) {$\mathbf{O}$};
\foreach \x in {A,B,C}
  \path[el,draw=black] (O) -- (\x);
\node[elo] at (4.6,1.5) (3312) {$\mathbf{v}$};
\node[elo] at (3.7,3.9) (132) {$\mathbf{u}$};

\foreach \x/\y/\z in {3.5/5.5/1,4.9/2.6/3,2.2/2.8/13,3.8/2.4/31,6.6/0.7/33,2.8/1/331}
  \node[halott] at (\x,\y) (\z) {};
\path[el,draw=red] (O) -- (1) -- (13) -- (132);
\path[el,draw=red] (O) -- (3) -- (33) -- (331) -- (3312);
\path[el,draw=red] (3) -- (31);

\foreach \x in {1,13,31,3,331,33}
  \path[el,draw=blue] (A) -- (\x);
\foreach \x in {3,33}
  \path[el,draw=blue] (C) -- (\x);
\foreach \x in {13,132,31}
  \path[el,draw=blue] (O) -- (\x);
\path[el,draw=blue] (B) -- (1);
\foreach \x in {3,33}
  \path[el,draw=blue] (3312) -- (\x);
\path[el,draw=blue] (1) -- (132);
\path[el,draw=blue] (3) -- (331);
\begin{scope}[xshift=8.7cm, yshift=6cm]
\node[right, text width=6cm] at (0,0) {{Coding of vertices\\ grouping of edges}};
\path[el, draw=black] (0.5,-3.5) -- (1.1,-3.5) node[right, black, text width=3.8cm] {\phantom{x} initial graph};
\path[el, draw=red] (0.5,-4.5) -- (1.1,-4.5) node[right, black, text width=3.8cm] {\phantom{x} forward edges};
\path[el, draw=blue] (0.5,-5.5) -- (1.1,-5.5) node[right, black, text width=3.8cm] {\phantom{x} shortcut edges};
\node[elo, right] at (0.5,-1.5) {$\mathbf{u}$};
\node[right] at (1.2,-1.5) {$\mathbf u=132$};
\node[elo, right] at (0.5,-2.5) {$\mathbf{v}$};
\node[right] at (1.2,-2.5) {$\mathbf v=3312$};
\end{scope}
\begin{scope}[xshift=18cm, yshift=5.7cm]
\path[draw=black, ultra thick] (-3,0.75) node[halott] (c) {} -- (1,1.5) node[halott] (b) {} -- (4,0.75) node[halott] (a) {} -- (c);
\node[elo] at (0,0) (o) {$\mathbf{O}$} [level distance=13mm, draw=red, ultra thick, sibling distance=25mm]
  child{node[halott] (1) {} child[missing] child[missing] child[sibling distance=15mm]{node[halott] (13){} child{node[elo] (132) {$\mathbf{u}$}} } }
  child[missing]
  child{node[halott] (3) {} child[sibling distance=15mm]{node[halott] (31){} } child[missing]
                            child[sibling distance=15mm]{node[halott] (33){} child{node[halott] (331){} child{node[elo] (3312) {$\mathbf{v}$}} } child[missing] } };

\foreach \x in {1,13,31,3,331,33}
  \path[draw=blue, very thick, dashed] (a) -- (\x);
\foreach \x in {3,33}
  \path[draw=blue, very thick, dashed] (c) -- (\x);
\foreach \x in {13,132,31}
  \path[draw=blue, very thick, dashed] (o) -- (\x);
\path[draw=blue, very thick, dashed] (b) -- (1);
\foreach \x in {3,33}
  \path[draw=blue, very thick, dashed] (3312) -- (\x);
\path[draw=blue, very thick, dashed] (1) -- (132);
\path[draw=blue, very thick, dashed] (3) -- (331);

\foreach \x in {a,b,c}
  \path[draw=black, ultra thick] (o) -- (\x);
\end{scope}

\end{tikzpicture}\label{fig:RANasTree}
\caption{Tree like structure of a realisation of $\RAN_2(8)$}
\end{figure}
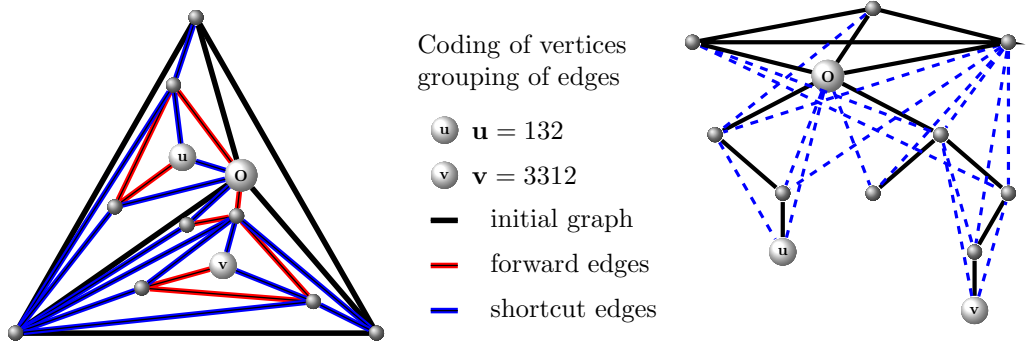

The grouping of the edges reveals the tree like structure of RANs. Interpreting the initial graph as the root, the forward edges are the \emph{edges of the tree}: along them we can go deeper down in the hierarchy of the graph. The shortcut edges \emph{only run along a tree branch:} between vertices that are in the same line of descent, so we can `climb up' to the root faster along these edges, see Figure 2.

\subsection{Distances in RANs and EANs: the main idea}
The decomposition above of edges of $\RAND(n)$ makes it possible to analyse distances in the graph.
Suppose we want to determine the distance between two arbitrary vertices $u,v$ with codes $\mathbf u, \mathbf v$. Then we can do it as follows:

First, determine the generation of their deepest common ancestor $\mathbf u\wedge \mathbf v$. Then see, how deep $\mathbf u$ and $\mathbf v$ are in the tree from $\mathbf u\wedge \mathbf v$, i.e. determine the length of their code below $\mathbf u\wedge \mathbf v$.
Finally, determine how fast can we reach the deepest common ancestor along the shortcut edges in these two branches, i.e. what is the minimal number of hops we can go up from $\mathbf u$ and $\mathbf v$  to $\mathbf u\wedge \mathbf v$?

If we pick $u,v$ uniformly at random, then we have to determine the \emph{typical length of codes} in the tree and the \emph{typical number of shortcut edges} needed to reach the typical common ancestor. If on the other hand we want to analyse the diameter or the flooding time, we have to find a `long' branch with `many' shortcut edges.
Clearly, one can look at the vertex of \emph{maximal depth} in the tree: but then - by an independence argument about the characters in the code and the length of the code -  with high probability the code of the maximal depth vertex in the tree will show \emph{typical behaviour} for the number of shortcut edges. On the other hand, we can calculate how many \emph{slightly shorter branches} are there in the tree. Then, since there are many of them, it is more likely that one of them has a code with more shortcut edges needed than typical.
Hence, we study the typical depth and also how many vertices are at larger, atypical depths of a branching process that arises from the forward edges of $\RAN$s. The effect of the shortcut edges on the distances is determined using renewal theory (also done in \cite{Albenque:2008:hopinRAN}) and large deviation theory.
 Finally, we optimise parameters such that we achieve the maximal distance by an entropy vs energy argument.

\subsection{Combinatorial analysis of shortcut edges}
Now we investigate the effect of shortcut edges on this tree.
By symmetry it is not hard to see that both a uniformly picked individual in the tree has a code where at each position, each character is uniform in $\Sigma_d$ and further, for two uniformly picked vertices with codes $\mathbf u$ and $\mathbf v$ the characters in these codes after the position $|\mathbf u \wedge \mathbf v|+1$ are again uniformly distributed.
Lemma \ref{lemma::neighbors} part (a) says that the shortcut edges of a vertex $\mathbf u$ in the tree are leading exactly to $T_i \mathbf u$, the prefixes of $\mathbf u$ achieved by deleting every character after the last occurrence of character $i$ in the code.
Recall that $P_i \mathbf u$ denotes the postfix of $\mathbf u$ that starts at the last occurrence of the character $i\in\Sigma_d$ in the code of $\mathbf u$.
Hence, let us denote the operator that gives the prefix with length $\min_i |T_i \mathbf u|$ by $T_{\min}$ and the length of the maximal cut by
\begin{equation}\label{def::Y_dRAN}
Y_d^{\scriptscriptstyle{AN}}(\mathbf u):=|\mathbf u|- | T_{\min} \mathbf u| = \max_{i\in \Sigma_d} \{ |P_i \mathbf u| \}.
\end{equation}
This is the length of the maximal hop we can achieve from the vertex $\mathbf u$ upwards in the tree via a shortcut edge.

Consecutively using the operator $T_{\min}$ we can decompose $\mathbf{u}$ into independent blocks, where each block, when reversed, ends at the first appearance of all the characters in $\Sigma_d$. We call such a block \emph{full coupon collector block}. E.g.\ for $\mathbf{u}=113213323122221131$ this gives $1|132|1332|31222\big|21131$.
Let us denote the total number of blocks needed in this decomposition by
\begin{equation}\label{eq:total-blocks-in-u}
N(\mathbf u)= \max\{k+1: T_{\min}^{k} \mathbf u \neq \varnothing \}.
\end{equation}
Note that this is not the only way to decompose the code in such a way that we always cut only postfixes of the form $P_i \mathbf u$: e.g.\  $1|132|1332|31222 211\big|31$ gives an alternative cut with the same number of blocks.

The following (deterministic) claim establishes that the decomposition along repetitive use of $T_{\min}$ (maximal possible hops) is optimal.
\begin{claim}\label{cl:optimal} Suppose we have an arbitrary code $\mathbf u$ of length $n$ with characters from $\Sigma_d$, that we want to decompose into blocks in such a way that \emph{from right to left}, each block ends at the first appearance of some character in that block.  Then, the minimal number of blocks needed is given by $N(\mathbf u)$.
\end{claim}

\begin{proof} Reverse the code to be read from left to right, and call the end of the blocks \emph{barriers}. Denote the position of the $i$-th barrier by $B_i$ in a decomposition, that does not use full coupon collector blocks all the time, and let $B_i^\star$ stand for the same in the decomposition with full coupon collector blocks. We show that $B_i\le B_i^\star$ , which, combined with the fact that the total number of blocks needed is $\max\{ i+1:  B_i<n \}$, implies the claim.
For the first block, clearly $B_1\le B_1^\star$ holds. If $B_1=B_1^\star$, we can cut the first block and compare the next one: let us denote by $j$ the first index that $B_j<B_j^\star$: now the block between $B_j$ and $B_{j+1}^\star$ is \emph{longer} than the block between $B_{j}^\star, B_{j+1}^\star$, so  two things can happen for $B_{j+1}$: either the block between $B_j$ and $B_{j+1}^\star$ is still an allowed block, and then we can put $B_{j+1}=B_{j+1}^\star$ or it is not, then we have to move the barrier $B_{j+1}$ forward, i.e. $B_{j+1}<B_{j+1}^\star$. From here we can proceed similarly by induction on $k$ for $j+k$.
\end{proof}
For a code of length $k$ with i.i.d. uniform characters from $\Sigma_d$ let
\begin{equation}\label{def::N_k}
H_k:= \max\{i: \sum\nolimits_{j=1}^i Y_d^{(i)} \le k\},
\end{equation}
where $Y_d^{(i)}$ are i.i.d copies of $Y_d$ in \eqref{def:y}, i.e.\ $H_k$ is the number of consecutive occurrences of full coupon collector blocks in the code.
\begin{lemma}\label{lem:graph-to-code} Suppose $\mathbf u$ is a code of length $k$ with uniform random characters from $\Sigma_d$ at each position.  Then
\[ \ba Y_d^{\scriptscriptstyle{AN}}(\mathbf u)\  &{\buildrel d \over = }\  Y_d\wedge k, \\
N(\mathbf u)\  &{\buildrel d \over = }\ H_k.
\ea\]
\end{lemma}
\begin{proof}
The last occurrence of any character $i\in \Sigma_d$ in a uniform code is the \emph{first occurrence from backwards}  of the same character. Hence, reverse the code of $u$, and then $|P_i \mathbf u|$ is the position of the first occurrence of character $i$ in a uniform sequence of characters, truncated at $k$, since $|\mathbf u|=k$.  Maximizing this over all $i\in \Sigma_d$ we get the well-known coupon collector problem, that has distribution $Y_d$.
For the second part, since $N(\mathbf u)$ cuts down full coupon collector blocks from the end of the ode of $u$ consecutively, the maximal number of cuts possible is exactly the number of consecutive full coupon collector blocks in the reversed code of $\mathbf u$, an i.i.d. code of length $k$. Since the length of each block has distribution $Y_d$, and they are independent, the statement follows.
\end{proof}

Recall $\mu_d, \sigma_d^2$ from \eqref{def:mud}. From basic renewal theory \cite{Fellerkonyv}
the following  central limit theorem holds
\begin{equation}\label{eqn::H_nCLT}
\frac{H_k-k/\mu_d}{\sqrt{k\sigma^2_d/\mu_d^3}}\toindis \mathcal{N}(0,1).
\end{equation}

\section{Distances in $\RAN$s and $\EAN$s}\label{sec::ProofDist}
Now we make use of the previous section and prove Theorems \ref{thm::MainRes}, \ref{thm::Diam} and \ref{thm::MainResEAN}.
Recall that the endpoints of a forward edge $u \to uk$ corresponds to two vertices $u$ and $uk$ ($k\in\Sigma_d$) where the code of $u$ can be obtained by dropping the last symbol of $uk$. When taking into account only the forward edges of a $\RAN$, there is a natural bijection between the vertices of the $\RAN$ and the nodes of a continuous time branching process (CTBP) \cite{Athreya_BP}, or a Bellman-Harris process.

Namely, consider a CTBP where the offspring distribution is deterministic: each individual has $d+1$ children and the lifespan of each individual is i.i.d.\  exponential with mean one. Thus, after birth a node is active for the duration of its lifespan, then splits, becomes inactive and at that instant gives birth to its $d+1$ offspring that become active for their i.i.d.\ $\mathrm{Exp}(1)$ lifespan.

The bijection between the CTBP at the split times and a $\RAN_d$ is the following: the individuals that have already split in the BP are the vertices already present in the $\RAN$, while the active individuals in the BP correspond to the active cliques in the $\RAN$. This holds since in a $\RAN$ at every step $d+1$ new active cliques arise in place of the one which becomes inactive. Furthermore, in a $\RAN$ an active clique is chosen uniformly at random in each step which is  -- by the memoryless property of exponential variables -- equivalent to the fact that the next individual to split in the BP is a uniformly chosen active individual.

We aim to prove Theorem \ref{thm::MainRes} first.
Now that the tree-structure of $\RAN_d(n)$ revealed to form a CTBP,  we first investigate the distance from the root in this tree.
We write $G_m$ for the \emph{generation} of the $m$-th splitting vertex in the BP, i.e.\ its graph distance from the root. The next lemma describes the typical size of $G_m$ as well as the \emph{degree of relationship} of two uniformly picked active individuals in the BP.

\begin{lemma}\label{lemma::asympOfG_m}
Let $Z$ denote a standard normal random variable. Then as $m\to\infty$
\begin{align}
\frac{G_m-\frac{d+1}{d}\log m}{\sqrt{\frac{d+1}{d}\log m}} \toindis Z.  \label{eq:CLT-single}
\end{align}
Further, let $G_{U}, G_{V}$ denote the generation of two independently and uniformly picked \emph{active vertices} in the BP after the $m$th split, and let us write $G_{U\wedge V}$ for the generation of the latest common ancestor of $U,V$. Then the marginal distribution $G_U\ {\buildrel d \over =}\  G_{m+1}$, and
\be  \left(\frac{G_{U}-G_{U\wedge V}-\frac{d+1}{d}\log m}{\sqrt{\frac{d+1}{d}\log m}}, \frac{G_{V}-G_{U\wedge V}-\frac{d+1}{d}\log m}{\sqrt{\frac{d+1}{d}\log m}}\right) \toindis (Z,Z'),
    \label{eq:CLT-joint}
    \ee
    where $Z,Z'$ are independent standard normal distributions.
\end{lemma}
\begin{proof} The lemma follows from the \emph{proof of} the CLT for the \emph{degree of relationship} in B\"uhler \cite[Theorem 3.3]{Buhler}.
The distribution of the generation of the $m$-th chosen vertex $G_m$ is the sum of independent Bernoulli random variables (see \cite{Buhler} or \cite{RvdH_FPP}). More precisely,
\begin{equation}\label{eqn::distofG_m}
G_m\stackrel{\mathrm{d}}{=}\sum_{i=1}^m\ind_i, \text{ where } \Prob{\ind_i=1}=\frac{d+1}{d i +1},
\end{equation}
The indicator $\ind_{i}$ is $0$ or $1$ whether in the line of ancestry of the $m$-th splitting vertex, the individual at the $k$-th split is surviving from time $k-1$ or newborn. From this identity it is easy to calculate the expectation and variance of $G_m$:
\begin{equation}\label{eqn::Exp(G_m)}
\Expect{G_m}=\frac{d+1}{d}\log m, \quad \quad \Var{G_m}=\Expect{G_m}+O\left(m^{-1}\right).
\end{equation}
The central limit theorem \eqref{eq:CLT-single} holds for the standardization of $G_m$ since the collection of Bernoulli random variables $\{\ind_i\}_{i=1}^m, \;m=1,2,\ldots$ satisfy Lindeberg's condition.

For the second statement, $G_U\ {\buildrel d \over =}\  G_{m+1}$ is immediate from the fact that the $m+1$th vertex to split in the process is a uniformly picked vertex among the active vertices.

Next, let us write $\tau_{U\wedge V}$ for the time when the latest common ancestor of $U$ and $V$ splits.
Then, conditioned on $\tau_{U\wedge V}$ the following two variables are independent and their joint distribution can be written as
\be\label{eq:degree-of-rel} (G_U\!-\!G_{U\wedge V}\!, G_V\!-\!G_{U\wedge V}\!)\ {\buildrel d \over =}\  \left(\sum_{i=\tau_{U\wedge V}}^{m+1} \ind_{i}, \sum_{i=\tau_{U\wedge V}}^{m+1} \ind_{i}'\right), \ee
where
\[\ba \Pv&\left((\ind_i, \ind_i')=(1,0) | \tau_{U\wedge V}< i \right)= \frac{d+1}{di+1}\frac{i-1}{i}\\
\Pv&\left((\ind_i, \ind_i')=(0,1)| \tau_{U\wedge V}< i \right)=\frac{d+1}{di+1}\frac{i-1}{i}\\
\Pv&\left((\ind_i, \ind_i')=(1,1), \tau_{U\wedge V}=i | \tau_{U\wedge V}\le i \right)= \frac{(d+1)}{(d i +1)i}\ea\] and conditioned on $\tau_{U\wedge V}$, different indices are independent. To see that $\tau_{U\wedge V}$ has a limiting distribution we can use \cite[Lemma 3.3]{Buhler}:
\be\label{eq:tau-dist} \Pv( \tau_{U\wedge V}\le k) = \prod_{i=k+1}^m (1-\Pv( \tau_{U\wedge V}=i |\tau_{U\wedge V}\le i), \ee
where the factors on the right hand side are the probabilities that the two ancestral lines do not merge at the $i$-th split. This is tending to a proper limiting distribution since
 \[ \sum_{i=1}^\infty \Pv( \tau_{U\wedge V}=i |\tau_{U\wedge V}\le i)=\sum_{i=1}^\infty \frac{d+1}{(di+1)di}\le \infty.\]
 Hence, $\tau_{U\wedge V}$ has a limiting distribution, and clearly $(m-\tau_{U\wedge V})/m\to 1$. From here, one can show the joint convergence of \eqref{eq:degree-of-rel} using Lindeberg CLT for linear combinations of $\sum_{i=\tau_{U\wedge V}+1}^{m+1} (\alpha \ind_{i}+\beta \ind_{i}')$ and get that the two variables in \eqref{eq:CLT-joint} tend jointly to a 2 dimensional standard normal variable.
\end{proof}

Now we are ready to prove Theorem \ref{thm::MainRes}.
\begin{proof}[Proof of Theorem \ref{thm::MainRes}]
Pick two uniform active cliques $u,v$ in the graph.
We write $\mathbf u, \mathbf v$  for the codes of these cliques and the corresponding vertices, $G_{U}=|\mathbf u|, G_{V}=|\mathbf v|$ for their generation.
As before, we write $\mathbf u\wedge \mathbf v$ for the latest common ancestor, i.e. their longest common prefix. Let us define the distinct postfixes after $\mathbf u\wedge \mathbf v$ by
\begin{equation*}\label{def:postfixes}
\mathbf u= (\mathbf u\wedge \mathbf v)\widetilde  {\mathbf u}, \quad \mathbf v= (\mathbf u\wedge \mathbf v) \widetilde {\mathbf  v}.
\end{equation*}
Note that $|\mathbf u\wedge \mathbf v|=G_{U\wedge V}, \ |\widetilde  {\mathbf u}|=G_U-G_{U\wedge V}, \ |\widetilde  {\mathbf v}|=G_V-G_{U\wedge V}$.
 By Lemma \ref{lemma::neighbors} and Claim \ref{cl:optimal} the length of the shortest path between $u,v$ satisfies:
\begin{equation*}\label{eq:distance-formula}
\dist(u,v)= N(\widetilde  {\mathbf u})+N( \widetilde {\mathbf v}),
\end{equation*}
and Lemma \ref{lemma::asympOfG_m} describes the typical distance between $\mathbf u$ and $\mathbf v$
along the tree. Note that except the first character, the characters in the codes $\widetilde{\mathbf{u}}$ and  $\widetilde{\mathbf{v}}$ are i.i.d.\ uniform on $\Sigma_d$.
Hence, by Lemma \ref{lem:graph-to-code} combined with  \eqref{eqn::Exp(G_m)}
\begin{equation*}
\Expect{H_{G_U-G_{U\wedge V}}} = \Expect{\Pv(H_{G_U-G_{U\wedge V}}| G_U-G_{U\wedge V})} = \frac{1}{\mu_d}\frac{d+1}{d}\log m(1+o(1)).
\end{equation*}
To obtain a central limit theorem for $H_{G_U-G_{U\wedge V}}$ observe that
\be\label{eq:HG-CLT} \ba  \frac{H_{G_U-G_{U\wedge V}}-\frac{1}{\mu_d}\frac{d+1}{d}\log m}{\sqrt{\frac{d+1}{d}\log m\,\sigma_d^2/\mu^3_d}}
&=\frac{H_{G_U-G_{U\wedge V}}-\frac{1}{\mu_d}(G_U-G_{U\wedge V}) }{\sqrt{(G_U-G_{U\wedge V})\sigma^2_d /{\mu^3_d}}} \cdot \sqrt{\frac{G_U-G_{U\wedge V}}{\frac{d+1}{d}\log m}}  \\
&\ \ \ + \frac{\frac{1}{\mu_d}(G_U-G_{U\wedge V})-\frac{1}{\mu_d}\frac{d+1}{d}\log m}{\sqrt{\frac{d+1}{d}\log m\,\sigma^2_d/\mu^3_d}}. \ea\ee
The first factor on the right hand side, conditionally on $G_U-G_{U\wedge V}$ with $G_U-G_{U\wedge V}\to \infty$, tends to a standard normal random variable independent of $G_U$ by the renewal CLT in \eqref{eqn::H_nCLT} and the second factor tends to one in probability by \eqref{eq:CLT-joint}. By \eqref{eq:CLT-joint} again, the second term tends to a $\mathcal{N}(0,\mu_d/\sigma^2_d)$. The two limiting normals are independent, thus
\begin{equation} \label{eqn::H_G_mCLT}
\frac{H_{G_U-G_{U\wedge V}}-\frac{1}{\mu_d}\frac{d+1}{d}\log m}{\sqrt{\frac{d+1}{d}\log m \,\sigma^2_d/\mu^3_d}} \toindis \mathcal{N}(0,1+\mu_d/\sigma^2_d).
\end{equation}
By conditioning first on $G_{U\wedge V}$, (as in the proof of Lemma \ref{lemma::asympOfG_m}) and using that the characters in the code of $\widetilde{\mathbf u}, \widetilde{\mathbf v}$ are all i.i.d. uniform in $\Sigma_d$, one can similarly show that  $(H_{G_U-G_{U\wedge V}}, H_{G_V-G_{U\wedge V}})$
tend jointly to two independent copies of $\CN(0, 1+\mu_d/\sigma^2_d )$ variables.
Using now that $\Hop(n)=H_{G_U-G_{U\wedge V}}+H_{G_V-G_{U\wedge V}}$, the first statement of the Theorem \ref{thm::MainRes} immediately follows by normalising such that the total variance is $1$.
The second statement follows by calculating how many active cliques a vertex with degree $k$ is contained:
a vertex with degree $d+1$ is contained in $d+1$ cliques, and when the degree of a node $v$ increases by $1$, then the number of cliques containing $v$ increases by $d-1$, thus a vertex with degree $k\geq d+1$ is contained in exactly
\be\label{eq::ean-active-k}
   A_k=k(d-1)+d^2-d+2
\ee
active cliques. It is also not hard to see that the total number of active clicks after $n$ steps is $A(n)=d n + d+2$.
This means, that picking two vertices $x,y$ according to the size-biased distribution given in \eqref{eq:biased} is equivalent that the uniformly picked active cliques $U,V$ are neighbouring these vertices. The distance between $x,y$ is then between $N(\widetilde  {\mathbf u})+N( \widetilde {\mathbf v})-2, N(\widetilde  {\mathbf u})+N( \widetilde {\mathbf v})$ since by Lemma \ref{lemma::neighbors} $x=T_i \mathbf u$ for some $i\in \Sigma_i$, hence we can gain at most $1$ hop by considering $x$ instead of the clique $U$ and the same holds for $y$ and $V$. Hence, the CLT for $U,V$ implies a CLT for two vertices picked according to the probabilities in \eqref{eq:biased}.
\end{proof}

\begin{remark}
\normalfont If we pick two vertices (x,y) of $\RAND(n)$ uniformly at random, then we should randomise the index $m$ of $G_m$ to be discrete uniform in $[1,2,\dots,m]$. One can easily see that in this case the CLT does not hold anymore. We still have
$\frac{\dist(x,y)}{2 \frac{d+1}{d}\log m}\toinp 1$.
\end{remark}

Since the proofs are similar, we prove Theorem \ref{thm::MainResEAN} next.
\begin{proof}[Proof of Theorem \ref{thm::MainResEAN}]
The proof follows analogous lines to the proof of Theorem \ref{thm::MainRes}, hence we give only the sketch.
The main idea here is that the tree can be viewed as a BP where at step $i$, each active individual splits with probability $q_i$ or stays active for the next step with probability $1-q_i$. Hence, Lemma \ref{lemma::asympOfG_m} can be modified as follows: the generation of a uniformly picked active individual at step $m$ satisfies
\[ \widetilde G_U\  {\buildrel d\over =}\ \sum_{i=1}^m \widetilde\ind_i,\]
where $\Pv(\widetilde\ind_i=1)=q_i$ and different indices are independent. This corresponds to following the ancestral line of $U$: in this ancestral line the generation is increased by $1$ in the $i$th step if the individual active at step $i$ split and stays if it did not split. Since splitting happens with probability $q_i$ at step $i$, the result follows and the CLT for $\widetilde G_U$ holds by Lindeberg CLT.
Now, for two uniformly picked individuals $U,V$ we have
\[ \left( G_U-G_{U\wedge V}, G_V-G_{U\wedge V}\right) {\buildrel d\over =} \left(\sum_{i=\tau_{U\wedge V}}^m \widetilde\ind_i, \sum_{i=\tau_{U\wedge V}}^m \widetilde\ind_i' \right),\]
where different indices are independent and conditioned on $\tau_{U\wedge V}$, $\widetilde\ind_i, \widetilde\ind_i'$ are independent indicators with $\Pv(\widetilde\ind_i=1)=\Pv(\widetilde\ind_i'=1)=q_i$. Since the variance $\sum_i q_i(1-q_i) \to \infty,$ the joint CLT follows in a similar manner then for Lemma \ref{lemma::asympOfG_m} if we can show that $\tau_{U\wedge V}$ has a limiting distribution. For this note that similarly as in \eqref{eq:tau-dist}
\be\label{eq:dist-of-tau-uv} \Pv(\tau_{U\wedge V}\le k)= \prod_{i=k+1}^m \left(1- \Pv\left(\tau_{U\wedge V}=i | \tau_{U\wedge V}\le i \right)\right)\ee
and the factors on the right hand side express that the two ancestral lines of $U,V$ do not merge yet at step $i$. Let us write $A(i)$ for the number of active vertices at step $i$. Then at step $i$ there are $Z_i:=\mathrm{Bin}(A(i), q_i)$ many vertices that split, each of them producing $d+1$ new active vertices, and hence the probability that the two ancestral lines merge at step $i$, conditioned on $A(i), A(i+1)$ equals
\be  \Pv\left(\tau_{U\wedge V}=i | \tau_{U\wedge V}\le i, A(i), A(i+1) \right)= \frac{Z_i (d+1) d}{(A(i+1) (A(i+1)-1)},\ee
where $A(i+1)=A(i)+Z_i (d+1)$, the new number of active vertices after the $i$th split.
If the sum in $i\in\N$ on the right hand side in the previous display is a.s. finite then \eqref{eq:dist-of-tau-uv} ensures that $\tau_{U\wedge V}$ has a proper limiting distribution. Hence we aim to show that this is the case whenever
the total number of vertices $N(n)\to\infty$. That is,
\[ \sum_{i=1}^\infty \frac{Z_i(d+1)d}{ A(i+1) (A(i+1)-1)} \le \infty \quad \text{ a.s. on }\{N(n)\to \infty\}.\]
Since $A(i+1)=N(i+1) d + d+1$,  and $Z_i=N(i+1)-N(i)$, we can approximate the above sum by
\[ (d+1) \sum_{i=1}^\infty \frac{d(N(i+1)-N(i))}{ (d N(i+1))^2 }.  \]
Now we can interpolate $N(i)$ with a continuous function and then this sum can be approximated by the integral
\[\int_{1}^\infty \frac{N'(x)}{N(x)^2}\mathrm dx < \infty\]
as long as $N(n)\to \infty$. Note that as long as $\sum_{n\in \N}q_n=\infty$, this is the case by the second Borel-Cantelli lemma: in each step we add at least a new node with probability $q_n$.
 The CLT then for the distances follows in the exact same manner as in the proof of Theorem \ref{thm::MainRes}.
\end{proof}

Next we aim to prove Theorem \ref{thm::Diam}, but we need some preliminary lemmas first. Lemma \ref{lemma::tailofG_m} shows that the tail of $G_m$ basically behaves as the the tail of the sum of $\log m$ i.i.d. Poi($(d+1)/d$) random variables and gives the depth of the longest branch in our BP, while Lemma \ref{lemma::LDforH_n} gives a large deviation bound for $H_k$ (recall \eqref{def::N_k}).
\begin{lemma}\label{lemma::tailofG_m}
The exact asymptotic tail behaviour of $G_m$ is given by
\begin{equation}\label{eqn::P(G_m>clogm)}
\lim_{m\to\infty}\frac{\log \left(\Prob{G_m>c\log m} \right)}{\log m}= c-\frac{d+1}{d}-c\log\left(\frac{d}{d+1}c\right)=:f_d(c).
\end{equation}
Further, in the branching process tree having $m$ vertices, the deepest branch satisfies
\begin{equation}\label{eqn::maxG_itoc_d}
\frac{\max_{i\leq m} G_i}{\log m} \toinp \tilde c_d,
\end{equation}
where $\tilde c_d:= \{    c_d>(d+1)/d,\; f_d(  c_d)=-1 \}$ (as defined in Theorem \ref{thm::Diam}).
\end{lemma}

\begin{proof} Let us calculate the following:
\begin{equation*}
 \frac{1}{\log m} \log \Expect{e^{\theta G_m}} = \frac{1}{\log m}\sum_{i=1}^m\log\Big(1+\frac{d\!+\!1}{d i+1}(e^\theta-1)\Big).
\end{equation*}
From the series expansion of $\log(1+x)$ we can see that
\begin{equation*}
\lim_{m\to\infty} \frac{1}{\log m} \log \Expect{e^{\theta G_m}} = \frac{d+1}{d}(e^\theta-1),
\end{equation*}
which is the cumulant generating function of a $\xi=\mathrm{Poi}((d+1)/d)$ random variable. The rate function of such a random variable is $-f_d(c)$. Hence, by the G\"{a}rtner-Ellis theorem for $\xi^{(i)}\sim \xi$ i.i.d.
\begin{equation*}
\lim_{m\to\infty} \frac{\log\left(\Prob{G_m>c\log m}\right)}{\log m}= \lim_{m\to\infty}\frac{ \log\left(\Prob{\sum_{i=1}^{\log m}\xi^{(i)}>c\log m}\right) }{\log m}= f_d(c).
\end{equation*}
As for \eqref{eqn::maxG_itoc_d}, since $f_d(\frac{d+1}{d})=0$, choosing any $c>\tilde c_d$ results in a summable bound on \eqref{eqn::P(G_m>clogm)}. Thus by the Borel-Cantelli lemma, for any $c>\widetilde c_d$ there are only finitely many $m$ such that the event $\{G_m>c\log m\}$ holds, giving the whp upper bound $\tilde c_d \log m$ on the depth of the BP. The lower bound can be obtained based on work of Broutin and Devroye \cite{Broutin:2006:LargeDev} since our BP is a special case of so-called random lopsided trees \cite{Choi:2001:lopsided, Kapoor:1989:lopsided}.
\end{proof}

\begin{remark}\label{remark::G_mtoinp}
\normalfont
The function $f_d(c)$ in \eqref{eqn::P(G_m>clogm)} is strictly negative outside $c=(d+1)/d$. This immediately implies the weaker result
$\frac{G_m}{\log m}\toinp \frac{d+1}{d}$.
\end{remark}

\begin{lemma}\label{lemma::LDforH_n}
For $1\leq\beta\leq \mu_d/(d+1)$, the number of consecutive occurrences of full coupon collector blocks in a code of length $k$ satisfies the large deviation
\begin{equation}\label{eqn::H_n>nb/mu}
\lim_{k\to\infty}\frac{1}{k}\log\left(\Pv\Big(H_k>\frac{\beta}{\mu_d}k\Big)\right) =-\frac{\beta}{\mu_d}I_d\Big(\frac{\mu_d}{\beta}\Big),
\end{equation}
where the large deviation rate function $I_d(x)$ of $Y_d$ was defined in \eqref{def:rate-function}.
\end{lemma}

\begin{proof}
Let $Y_d^{(i)}$ be i.i.d. distributed according to $Y_d$. Since
\begin{equation*}
\Prob{H_k>\frac{\beta}{\mu_d}k} = \Pv\Big(\sum_{i=1}^{k\beta/\mu_d} Y_d^{(i)} < \left(\frac{\beta}{\mu_d}k\right)\cdot\frac{\mu_d}{\beta}\Big),
\end{equation*}
we can apply Cram\'er's theorem to obtain \eqref{eqn::H_n>nb/mu}.
\end{proof}

\begin{remark}\label{remark::LDfuncofY}
\normalfont
The rate function $I_d(x)$ has no explicit form. It can be computed numerically from
\begin{equation*}
I_d(x)=\lambda^*(x)\cdot x - \log\Expect{e^{\lambda^*(x) Y_d}},
\end{equation*}
where $\lambda^*(x)$ is the unique solution to the equation
$\frac{\partial}{\partial\lambda} \log\Expect{e^{\lambda Y_d}} =x$ and
\begin{equation*}
\log\Expect{e^{\lambda Y_d}} = \log d!-d\log(d+1)+(d+1)\lambda -\sum_{i=1}^d \log\Big(1-\frac{i}{d+1}e^{\lambda}\Big).
\end{equation*}
\end{remark}
Now we proceed to the proof of Theorem \ref{thm::Diam}.

\begin{proof}[Proof of Theorem \ref{thm::Diam}]
First, we have seen at the end of the proof of Theorem \ref{thm::MainRes} that switching from vertices to neighbouring active cliques only changes the distances by $2$, hence we rather investigate the diameter of the  vertices coded by active cliques. Let us denote the set of active cliques at step $m$ by $\CA(m)$. We have seen that $|\CA(m)|=dm+d+1$.
Our aim is to estimate the number of active cliques $u$ whose codes are relatively long i.e.\ at least $\alpha\tilde c_d\log m$ for $\alpha\in(0,1]$, and in their code $N(\mathbf u)$, the number of hops is larger than expected, i.e.\ at least $n\beta/\mu_d$ for $\beta\in[1,\mu_d/(d+1)]$. Define the indicator variables for each clique $u\in \CA(m)$
\begin{equation*}
J_u=\ind\big[G_{u}>\alpha\tilde c_d\log m, \ H_{G_u}>\frac{\beta}{\mu_d}\alpha\tilde c_d\log m\big].
\end{equation*}
Summing over $u\in \CA(m)$, the expected number of such cliques is
\begin{equation*}
|\CA(m)|\Ev\Big[\tfrac1{|\CA(m)|}\sum_{u \in \CA(m)} J_u\Big]
= (dm+d+1) \Prob{G_u>\alpha\tilde c_d\log m}\Pv\Big[H_{G_u}>\frac{\beta}{\mu_d}\alpha\tilde c_d\log m\Big],
\end{equation*}
since the length of a code and its coordinates are independent, and the expected value on the left hand side corresponds to a uniformly picked active clique $U\in \CA(m)$. By Lemma \ref{lemma::asympOfG_m} $G_u\  {\buildrel d\over =} G_{m+1}$, hence \eqref{eqn::P(G_m>clogm)} and \eqref{eqn::H_n>nb/mu} implies that
\begin{equation}\label{eqn::E(sumJu)}
\Ev\Big[\sum_{u \in \CA(m)} J_u\Big] = m^{\left(1+f_d(\alpha \tilde c_d) -\alpha \tilde c_d\frac{\beta}{\mu_d}I_d\left(\frac{\mu_d}{\beta}\right)\right)(1+o(1))} =: m^{g(\alpha,\beta)(1+o(1))}.
\end{equation}
We wish to choose $(\alpha,\beta)$ so that $\Pv\big(\sum_{u \in \CA(m)} J_u>0\big)>0$. Thus necessarily $g(\alpha,\beta)\geq 0$. We will need a upper bound on the second moment
\begin{equation*}
\Ev\Big[\Big(\sum_{u \in \CA(m)} J_u\Big)^2\Big]= \sum_{u \in \CA(m)} \Ev[J_u]+\sum_{u,v\in \CA(m), u\neq v} \Ev[J_u J_v].
\end{equation*}
Using the proof of Lemma \ref{lemma::asympOfG_m}, we know that the joint distribution of two uniformly picked individuals satisfies that their common ancestor $G_{u\wedge v}$ has a limiting distribution and conditioned on the splitting time $\tau_{u\wedge v}$ of $u\wedge v$, the joint distribution of $G_u-G_{u\wedge v}, G_v-G_{u\wedge v}$ can be described as the sum of indicators. Further, the characters after $u\wedge v$ in both codes are independent and uniform in $\Sigma_d$. Hence, it is not hard to see that
\begin{equation}\label{eqn::upperforE(JuJv)}
\sum_{u,v\in \CA(m), u\neq v} \Ev[J_u J_v] \le \Ev[ \sum_{u\in \CA(m)} J_u ]^2.
\end{equation}
From a Cauchy-Schwarz inequality followed by \eqref{eqn::upperforE(JuJv)}
\begin{equation}\label{eq:CS}
\Pv\big(\!\!\sum_{u \in \CA(m)} J_u>0\big) \geq \frac{\Ev[ \sum_{u\in \CA(m)} J_u ]^2}{\Ev\Big[\Big(\sum_{u \in \CA(m)} J_u\Big)^2\Big]} \geq \frac{\Ev[ \sum_{u\in \CA(m)} J_u ]^2}{\Ev[ \sum_{u\in \CA(m)} J_u ]+\Ev[ \sum_{u\in \CA(m)} J_u ]^2}\,>0,
\end{equation}
if and only if $g(\alpha,\beta)\geq0$. From this it is immediate that for the correct order of magnitude of the diameter of the graph we need to pick the largest product $\alpha\beta$ with the property that $g(\alpha,\beta)\ge 0$. Since $g(\alpha, \beta)$ is decreasing both in $\alpha$ and in $\beta$,
if $g(\alpha,\beta)>0$ then we can increase one or both of the parameters $\alpha'\ge \alpha,  \beta'\ge \beta$ s.t. $\al'\beta'>\al\beta$ and $g(\alpha',\beta')=0$. Hence, the diameter can be achieved if we restrict the problem to $(\alpha,\beta)$ pairs so that the exponent $g(\alpha,\beta)$ equals zero. Thus we arrive at the maximization problem \eqref{eqn::defalpha*beta} in Theorem \ref{thm::Diam}:
\begin{equation*}
\max_{\alpha,\beta} \Big\{\frac{\tilde c_d}{\mu_d}\alpha\beta \;:\; g(\alpha,\beta) = 0,\; (\alpha,\beta)\in(0,1]\times[1,\frac{\mu_d}{d+1}] \Big\}.
\end{equation*}
We show that this maximization problem has a unique solution $(\widetilde\alpha,\widetilde \beta)$ in Lemma \ref{lemma::SolofMaxProblem} below.

Now apply the second moment method in \eqref{eq:CS} with the maximising $(\widetilde\alpha,\widetilde \beta)$ for a lower bound and any other $(\alpha', \beta'): \alpha'\beta'>\widetilde\alpha\widetilde \beta$ and Markov's inequality for an upper bound to finally conclude
\begin{equation}\label{eqn::maxH_Gu}
\frac{\max_{u\in \CA(m)}H_{G_u}}{\log m} \toinp \frac{\tilde c_d}{\mu_d}\tilde\alpha\!\cdot\!\tilde\beta.
\end{equation}

The statement of Theorem \ref{thm::Diam} for the flooding time now follows from the fact that if $u$ is a uniformly picked active clique, then Lemma \ref{lemma::asympOfG_m} implies that the CLT holds for its generation $G_u$, and since the characters are uniform in the code of $u$, similarly as in \eqref{eq:HG-CLT},  the CLT holds for $H_{G_u}$ as well.  Further, since in $\Flood(u,v)$ we maximise the distance over the choice of the other vertex $v$, clearly whp we can pick $v$ such that the latest common ancestor $u\wedge v$ is the root itself. This combined with the fact that the distance changes only by at most $2$ if we consider active cliques instead of vertices in the graph implies
\[ \Flood(n)\stackrel{d}{=}H_{G_u}+ \max_{v\in \CA(n)}H_{G_v} \]
and the statement of the theorem follows from the distributional convergence of $\frac{H_{G_u}}{\log m}\toindis \frac{d+1}{d \mu_d}$ and \eqref{eqn::maxH_Gu}.

For the diameter we have \[ \frac{\Diam(n)}{\log m}\stackrel{d}{=}2\frac{\max_{v\in \CA(n)}H_{G_v}}{\log m},\] since for any $\ve>0$, whp there are at least two vertices that are not closely related to each other and both satisfy $H_{G_v}/ \log m > \frac{\widetilde c_d}{\mu_d} \widetilde \alpha\widetilde\beta(1-\ve)$, but whp there are no vertices that satisfy $H_{G_v}/ \log m >  \frac{\widetilde c_d}{\mu_d} \widetilde \alpha\widetilde\beta(1+\ve).$
\end{proof}

We are left to analyse the maximization problem:
\begin{lemma}\label{lemma::SolofMaxProblem}
The maximization problem \eqref{eqn::defalpha*beta} has a unique solution $(\tilde\alpha, \tilde\beta)\in(0,1]\times[1,\frac{\mu_d}{d+1}]$, and further this solution satisfies
\begin{equation*}
\tilde\alpha = \frac{1}{\tilde c_d}\frac{d+1}{d} \exp\big\{ \!-\!I'_d(\mu_d/\tilde\beta)\big\},\quad
\frac{\tilde\beta}{\mu_d}I_d\Big(\frac{\mu_d}{\tilde\beta}\Big) = \frac{1+f_d(\tilde\alpha \tilde c_d)}{\tilde\alpha \tilde c_d}.
\end{equation*}
\end{lemma}

\begin{proof}
Define the Lagrange multiplier function $\CL(\alpha,\beta,\lambda):=\alpha\beta-\lambda g(\alpha,\beta)$. Necessarily the optimal $(\tilde\alpha,\tilde\beta)$ satisfies $\nabla \CL(\tilde\alpha,\tilde\beta,\widetilde \lambda)=0$. The partial derivative $\CL(\alpha,\beta,\lambda)'_\lambda=0$ simply gives the condition $g(\alpha,\beta)=0$. Further, the optimising $\widetilde \lambda$ can be expressed from $\CL(\alpha,\beta,\lambda)'_\alpha=0$ and $\CL(\alpha,\beta,\lambda)'_\beta=0$ and satisfies
\begin{equation*}
\widetilde \lambda=\frac{\beta}{\frac{\partial}{\partial\alpha}g(\alpha,\beta)} = \frac{\alpha}{\frac{\partial}{\partial\beta}g(\alpha,\beta)}.
\end{equation*}
After differentiation of $g(\alpha, \beta)=1+f_d(\alpha \tilde c_d) -\alpha \tilde c_d\frac{\beta}{\mu_d}I_d\left(\frac{\mu_d}{\beta}\right)$, rearranging terms and using that $f'_d(x)=-\log\big(\frac{d}{d+1}x\big)$ we obtain the first condition.
To check the sufficiency we look at the bordered Hessian
\begin{equation*}
\left[
  \begin{array}{ccc}
    0 & \frac{\partial g}{\partial \alpha} & \frac{\partial g}{\partial \beta} \\
    \frac{\partial g}{\partial \alpha} & \frac{\partial^2 \alpha\beta}{\partial \alpha^2 } & \frac{\partial^2 \alpha\beta}{\partial \alpha \partial \beta} \\
    \frac{\partial g}{\partial \beta} & \frac{\partial^2 \alpha\beta}{\partial \alpha \partial \beta} & \frac{\partial^2 \alpha\beta}{\partial \beta^2} \\
  \end{array}
\right] =
\left[
  \begin{array}{ccc}
    0 & \frac{\partial g}{\partial \alpha} & \frac{\partial g}{\partial \beta} \\
    \frac{\partial g}{\partial \alpha} & 0 & 1 \\
    \frac{\partial g}{\partial \beta} & 1 & 0 \\
  \end{array}
\right].
\end{equation*}
Its determinant is $\big(\frac{\partial^2 g(\alpha,\beta)}{\partial \alpha \partial \beta}\big)^2>0$, thus the condition is also sufficient.
We note that the solution can be approximated by numerical methods.
\end{proof}
\begin{remark}\normalfont We mention here the difficulties in the analysis of the diameter and flooding time of $\EAN$s: the main difficulty here is to understand the proper correlation structure of the codes (and shortcut edges) on the vertices of the BP: (a) The corresponding BP tree is fatter than the BP for $\RAN$ as soon as $n^{-1}=o(q_n)$. (b) In each step each vertex splits independently of the past with probability $q_n$. (a) and (b) together imply that even though the marginal distribution of the characters of a uniformly picked clique $U$ is uniform in $\Sigma_d$, still it is more likely that the 'neighbouring codes' are also present in the graph and hence codes for which $N(\mathbf u)$ is large are more likely to appear.  Hence we expect that the diameter will have a larger constant in front of $\sum q_i$ than the constant in front of $\log n$ for $\RAN$.  (Compare it to the diameter of the deterministic AN: with $q_n\equiv 1$ it is not hard to see that $\diam( \mathrm{AN}_{d}(n))=2n/(d+1))$.
\end{remark}

\section{Degree distribution of Apollonian networks}\label{sec::ProofDegDistr}
In this section we prove the results related to the degree distribution. We start with analysing $\RAN$s first.
\subsection{Sketch of Proof of Theorem \ref{739}}\label{sec::proofDegreeRAN}
The proof of Theorem \ref{739} determining the degree distribution of $\RAN$s consist of two main steps that are described in Lemmas \ref{706} and \ref{738}. The first lemma shows that $N_k(n)$ is getting close to its expectation uniformly in $k$ as $n\to \infty$. The method we describe here is an adaptation of  the standard martingale method and similar to that in \cite{BBCR,BRST,Remco}. Parallel to our work, Frieze and Tsourakakis \cite{Tsourakis:2012:wrongRANdiam} applied this method to show Lemma \ref{706} and \ref{738} for two dimension and their proof can be generalized to higher dimensions without any difficulty,
hence we only give a sketch of proof here.
\begin{lemma}{Frieze and Tsourakakis}\label{706} \cite{Tsourakis:2012:wrongRANdiam}
Fix $d\geq 2$ and $c_1>\sqrt{8}(d+1)$. Then
\[
  \Pv \left( \max_k \left| \widetilde N_k(n)-\Ev\big[\widetilde N_k(n)\big]\right| \geq c_1\sqrt{n\log n}\right)=o(1).
\]
\end{lemma}
This lemma tells us that $\widetilde N_k(n)$ is concentrating around its expected value. From this we immediately get the concentration of $\widetilde p_k(n)=\widetilde N_k(n)/(n+d+2)$ around its expected value. The second lemma approximates the difference between this expected value and $p_k$, the stationary distribution.

\begin{lemma}\label{738}\label{lem:degreeRAN}
There exists a probability distribution $\{ p_k\}_{k=d+1}^{\infty}$ for which for any $n\geq 0$ and for any $k\geq d+1$
\[
  \left| \Ev[\widetilde N_k(n)]- p_k (n+d+2) \right| \leq c_2\sqrt{n\log n}
\]
with some constant $c_2$. The distribution $\{p_k\}_{k\in \N}$ is determined in \eqref{k17} and it has a power-law asymptotic decay with exponent $\frac{2d-1}{d-1}\in (2,3]$ for $d\geq 2$.
\end{lemma}
As mentioned above, we do not give the proof of these lemmas here. The methods however are similar then to the ones used in the proof of Lemma \ref{lem:EAN1}, \ref{lem:EAN2} for the $\EAN$ below.
Given these two lemmas, the proof of Theorem \ref{739} follows:
\begin{proof}[Proof of Theorem \ref{739}]
  By triangle inequality Theorem \ref{739} immediately follows from Lemmas \ref{706} and \ref{738} with $c=c_1+c_2$.
\end{proof}

\subsection{Proof of Theorem \ref{k14}}\label{sec::proofDegreeEAN}

We prove Theorem \ref{k14} connected to $\EAN$s again in two main steps, as in the case of $\RAN$s.
Recall the definition $p_k(n)$ from \eqref{def:pk}. We denote the additional number of vertices after $n$ steps by $N(n)$, that is, $|V(n)|=N(n)+d+2$. Note that $|V(n)|$ is random, hence here it is better to have a conditional concentration result:
Let us denote the sigma algebra generated by $\{N(1), \dots, N(n)\}$ by $\CG_n$. The following lemma tells us that the empirical proportion of degree $k$ vertices is concentrating around its $\CG_n$-conditional mean:
\begin{lemma}\label{lem:EAN1}
Fix the dimension $d\geq 2$, a constant $c>\sqrt{8}(d+1)^{3/2}$, and a sequence of node arrival probabilities $\{q_n\}_{n=1}^\infty$ such that $N(n)\to \infty$ a.s.\ as $n\to \infty$. Then
\begin{equation} \label{k15}
  \Pv \left( \max_k \big|   p_k(n)-\Ev\big[  N_k(n)| \CG_n \big]/|V(n)|\big| \geq c\sqrt{ \frac{\log  N(n)}{N(n)}}\  \Big| N(n) \right)=o(1).
\end{equation}
\end{lemma}
The next lemma tells us that the $\CG_n$-conditional mean of the proportion of degree $k$ vertices is tending to $p_k$ given in \eqref{eq:pkdef}:
\begin{lemma} \label{k12}\label{lem:EAN2}
Let $d\geq 2$ and let us assume
  $q_n\to 0, \ N(n)\to \infty $ a.s. Then there exists constants $0<C_k<\infty$ for which for any $k\geq d+1$ and for any $\delta>0$
\[
  \left| \Ev\big[  N_k(n)| \CG_n\big]/|V(n)|-  p_k \right| \le C_k
  \Big(\sum_{i=1}^n \tfrac{N(i+1)-N(i)}{dN(i)}\Big)^{k-d}.
  \]
The distribution $p_k$ is the same as the asymptotic degree distribution of $\RAND$ given in \eqref{eq:pkdef}.
\end{lemma}
\begin{proof}[Proof of Theorem \ref{k14}]
The proof of the theorem is immediate from the triangle inequality and Lemma \ref{lem:EAN1}, \ref{lem:EAN2} once we establish that (a) $N(n)\to \infty$ and (b) the error term in Lemma \ref{lem:EAN2} is at most of order $\sqrt{\log N(n)}$.
For (a) note that in each step there is always at least one active clique, hence the number of new vertices added in step $i$ is bounded from below by an indicator variable that has success probability $q_i$.
Hence, by the second Borel Cantelli lemma, $N(n)\to\infty$ a.s. as long as $\sum_{n\in\N} q_n=\infty$.

For (b) we need to compare $N(n)$ to the order of the error terms in Lemma \ref{lem:EAN2}: the number of active cliques $A(n)=dN(n)+d+1$, hence the number of new vertices after the $n+1$th step satisfies $N(n+1)-N(n)=\mathrm{Bin}(dN(n)+d+1, q_n)$, with $N(1):=\mathrm{Bin}(d+1, q_1)$.
Now we can approximate $N(n)$ by a continuous function to estimate the error terms in Lemma \ref{lem:EAN2}:
\[
  \sum_{i=1}^n \frac{N(i+1)-N(i)}{dN(i)}\le \frac{C}{d}\int_1^n \frac{N'(x)}{N(x)} \le C \log N(n).
\]
and therefore for any fixed $k$
\[
  C_k\Big(\sum_{i=1}^n \tfrac{N(i+1)-N(i)}{dN(i)}\Big)^{k-d}\le \Big(C_k\log N(n)\Big)^{k-d}=o(\sqrt{\log N(n)}),
 \]
 where the constants $C, C_k>0$ can change along the lines.
Hence the statement of Theorem \ref{k14} follows by combining Lemmas \ref{lem:EAN1} and  \ref{lem:EAN2} by a triangle inequality.
\end{proof}
\begin{remark}[The order of magnitude of $|V(n)|$]
\normalfont It is elementary to show that the nonnegative martingale
\be \label{103} M_n'= N(n) \prod_{i=1}^{n-1} (1+d q_i)^{-1}\ee
is square integrable if $\sum_{n\in \N} q_n=\infty$ and so there exists a $\xi\ge 0$ random variable
\[
\frac{N(n)}{\prod_{i=1}^{n-1} (1+d q_i)}\to \xi \quad a.s.
\]
\end{remark}
\begin{claim}\label{k19}
  The series $p_k$ given in \eqref{eq:pkdef} is a probability distribution.
\end{claim}
\begin{proof}[Proof of Claim \ref{k19}]
Clearly $p_k\geq 0$. Combining the formula for $p_k$ in \eqref{eq:pkdef} with an elementary rewrite of the fraction of the Gamma-functions inside the sum yields
\[ \begin{aligned}
  \sum_{k=d+1}^{\infty}&{ p_k}=\frac{d}{2d+1}\frac{ \Gamma \left( 1\!+\!\tfrac{2d+1}{d-1}\right) }{\Gamma \left( 1\!+\!\tfrac{2}{d-1}\right) } \sum_{k=d+1}^{\infty}{\frac{\Gamma \left( k\!-\!d\!+\tfrac{2}{d-1}\right) }{\Gamma \left( k\!-\!d\!+\!\tfrac{2d+1}{d-1}\right) }}\\
  &=\frac{d}{2d\!+\!1}\frac{\Gamma \left( 1\!+\!\frac{2d+1}{d-1}\right) }{\Gamma \left( 1\!+\!\frac{2}{d\!-\!1}\right) }
  \cdot \sum_{k=d+1}^{\infty}{\frac{d\!-\!1}{d}\left( \frac{\Gamma \left( k\!-\!d\!+\!\frac{2}{d\!-\!1}\right) }{\Gamma \left( k\!-\!1\!-\!d\!+\!\frac{2d\!+\!1}{d\!-\!1}\right) }-
  \frac{\Gamma \left( k\!+\!1\!-\!d\!+\!\frac{2}{d\!-\!1}\right) }{\Gamma \left( k\!-\!d\!+\!\frac{2d\!+\!1}{d\!-\!1}\right)}\right) }=1,
\end{aligned} \]
since the last sum is telescopic.
\end{proof}


Now we prove Lemmas \ref{lem:EAN1} and \ref{lem:EAN2}. To do this, the following observations will be useful. By construction, there are $d+1$ active cliques at time $n=0$. When a new node is born the number of active cliques increases by $d+1-1=d$, thus at time $n$ there are
\be\label{eq::ean-active}
    A(n)=N(n)d + d +1
\ee
active cliques given $N(n)$, the number of non-initial nodes at time $n$. By the same argument as in \eqref{eq::ean-active-k}, a node with degree $k\geq d+1$ is contained in exactly
   $A_k=2+(k-d)(d-1)$
active cliques.

\begin{proof}[Proof of Lemma \ref{lem:EAN1}]
We prove the Lemma \ref{lem:EAN1} by using the Azuma\,--\,Hoeffding inequality in an elaborate way. 
Let us use the notation
$ V(n):=\sqrt{  N(n) \log   N(n)}$, 
and recall the sigma algebra $\CG_n$.
We aim to show that there exist a constant $c>0$ such that
\be\label{eq::eanaim}
  \Pv \left( \max_k \left|   N_k(n)-\Ev\big[  N_k(n)| \CG_n\big]\right| \geq c  V(n) \;\Big\vert\;  \CG_n\right)=o(1).
\ee
Taking conditional expectation w.r.t.\ $N(n)$ in \eqref{eq::eanaim} immediately gives Lemma \ref{lem:EAN1}. First note that at time $n$ the maximal degree of any vertex is $  N(n)+d-3$. Thus the left hand side of \eqref{eq::eanaim} is at most
\[ 
\sum_{k=d+1}^{  N(n)+d-3}\Pv\left( |   N_k(n)-\Ev[  N_k(n)|\CG_n]| \geq c  V(n) \mid  \CG_n\right ).
 \]
Since there are $  N(n)+d-3$ summables, it is enough to prove that \emph{uniformly in $k$} with $d+1\leq k\leq   N(n)+d-3$
\begin{equation} \label{713}
  \Pv \Big(|   N_k(n)-\Ev[  N_k(n)|\CG_n]| \geq c  V(n) \Big)=o\left(   N(n)^{-1}\right).
\end{equation}
For a fixed time $r$, let us fix an ordering of the cliques of the graph $\mathrm{EAN}_d(r,\{q_n\})$. Clearly, the number of active cliques $A(r)< (d+1)^r$. To get $\EAND(r+1)$ we draw an independent Bernoulli($q_{r+1}$) random variable for every clique in $\mathrm{EAN}_d(r,\{q_n\})$.
Hence,
for  $1\leq r\leq n$ and $0\leq s\leq A(r)$, it is reasonable to introduce   $\mathcal{F}_{r,s}$, the $\sigma$-algebra generated by the graph at time $r-1$ and the first $s$ coin flips at time $r$ and $\CG_n$. It is straightforward that $\CG_n=\mathcal{F}_{1,0}\subseteq \dots \subseteq \mathcal{F}_{1,d+1} \subseteq \mathcal{F}_{2,0} \subseteq \dots \subseteq \mathcal{F}_{n,A(n)}$. With this filtration, let us introduce the following Doob-martingale:
\[
   {M}_{r,s}=\Ev\left[  N_k(n) \mid \mathcal{F}_{r,s}\right],
\]
where $k$ is fixed. Clearly $ {M}_{1,0}=\Ev\left[  N_k(n) \mid  \CG_n\right]$, and $ {M}_{n,  A(n)}=  N_k(n)$. Now, we would like to estimate the difference between $ {M}_{r,s}$ and $ {M}_{r,s-1}$. We will see that
\begin{equation} \label{k11}
  \left|  {M}_{r,s}- {M}_{r,s-1}\right| \leq 2(d+1)\ \ \forall\  r\in \{ 1,\dots ,n\},\ \forall\  1\leq s<   A(r)\leq (d+1)^r.
\end{equation}
From the definition of $  M_{r,s}$, we see that the difference is caused by the extra information whether the $s$-th coin flip raises a new node or not. Let us consider the two different realizations, i.e.\ in $\EAN(r,s)_a$ there is a new node $v_{r,s}$ at the $s$-th coin flip and in $\EAN(r,s)_b$ it is not. Note that the number $  N(r+1)-  N(r)$ of new nodes at time $r$ is \emph{included in the $\sigma$ algebra} and therefore there must be an $s^\prime$ with $s\!<\!s^\prime\!<\!  A(r)$ that at the $s^\prime$th coin flip a new node $v_{r,s^\prime}$ will born in $\EAN(r,s)_b$ but not in $\EAN(r,s)_a$. Hence the graphs $\EAN(r+1,0)_a$ and $\EAN(r+1,0)_b$ might be coupled in such a way that \emph{the number of nodes are the same} and \emph{every node has the same degree  except} for the $d+1$ neighbors of $v_{r,s}$ in $\EAN(r,s)_a$ and the $d+1$ neighbors of $v_{r,s^\prime}$ in $\EAN(r,s)_a$. Since the degree of vertices that are born later than $r$ are not affected by what happens before time $r$, we can extend this coupling up to time $n$ such that there are at most $2(d+1)$ nodes with different degrees. Thus, taking expectation with respect to $\mathcal F_{r,s-1}$ conserves this difference, which implies \eqref{k11}.

We have just proved that $  M_{r,s}$ is a martingale with bounded increments. Observe that every new node will create $d+1$ new active cliques and then induce $d+1$ coin flips. Thus there are less than $(d+1)  N(n)$ coin flips until time $n$ and so thus $|  {M}_{r,s}- {M}_{r,s-1}| \neq 0$ only at most $(d+1)  N(n)$ times. Hence the Azuma\,--\,Hoeffding inequality gives us that
\[
  \Pv \left(|   N_k(n)-\Ev\big[  N_k(n)|  \CG_n\big]| \geq a \;\Big\vert\;  \CG_n\right)\leq 2\exp\Big\{ -\frac{a^2}{8  N(n)(d+1)^3}\Big\}.
\]
Now set $a=c  V(n)$, $c>\sqrt{8}(d+1)^{3/2}$: 
\[
\Pv \bigg(\!\Big|   N_k(n)-\Ev\big[  N_k(n)|\CG_n\big]\Big|\! \geq\! c  V(n) \mid \CG_n\!\bigg)\! \leq 2   N(n)^{\!-\tfrac{c^2}{24(d+1)^2}}\!\! \leq o\! \left(   N(n)^{-1} \right).
\]
Note that this bound is \emph{uniform in $k$}, hence \eqref{713} and \eqref{eq::eanaim} follows.
\end{proof}
 \begin{proof}[Proof of Lemma \ref{k12}]
 We aim to write a recursion for $\Ev[N_k(n)|\CG_n]$. Note that $A(n),N(n+1)-N(n)\in \CG_{n+1}.$
 Let us introduce the $n$th empirical occupation parameter:
 \[ \hat q_n:= \frac{N(n+1)-N(n)}{d N(n) + d+1 }= \frac{1}{A(n)}\sum_{i=1}^{A(n)} \ind_{ \{\text{the $i$th triangle is filled}\} }. \]
 Since we assume $N(n)\to \infty$, $ \hat q_n=q_n(1+o(1))$ a.s.
Given that there are $N(n+1)-N(n)$ successes in $A(n)$ Bernoulli trials, the places of these successful trials are uniformly distributed.
Clearly, $N_k(n)$ can change in three different ways in the $n+1$-th step:
\begin{enumeratei}
	\item A new node can connect to a vertex with degree $k$. Since a vertex with degree $k$ is contained in $ A_k$ active cliques (see \eqref{eq::ean-active}) the degree stays $k$ from step $n$ to $n+1$ with this happens with $\CG_n$-conditonal probability  $(1-\hat q_n)^{  A_k}$.
	\item A degree of a vertex can increase to $k$. A vertex with degree $k-\ell$, $\ell=1,\dots,\frac{d-1}{d}k$ is contained in $  A_{k-\ell}$ many active cliques, hence this happens with $\CG_n$-conditonal probability $\binom{  A_{k-\ell}}{\ell}\hat q_n^{\ell}(1-\hat q_n)^{  A_{k-\ell}-\ell}$.
	\item When $k=d+1$, $N_{d+1}(n)$ grows by the number of new nodes $N(n+1)-N(n)$
	\end{enumeratei}
 Hence we can write the following conditional recursion:
 \be\label{eq:conditional-recursion}\ba
\Ev&[N_k(n+1)|\CG_{n+1}]= \Ev[N_k(n)| \CG_n ] \left(1-\hat q_n\right)^{A_k} + \Ev[N_{k-1}(n)| \CG_n] A_{k-1} \hat q_n (1-\hat q_n)^{A_{k-1}-1} \\
& + \sum_{\ell=2}^{(d-1)k/d}  \Ev[N_{k-\ell}(n)| \CG_n] { A_{k-\ell} \choose \ell } \hat q_n^\ell (1-\hat q_n)^{A_{k-\ell}-\ell} + \ind_{\{k=d+1\}} (N(n+1)- N(n))
\ea
\ee
Now, we first try to find the `stationary solution' of this recursion in the form $\Ev[N_k(n)|\CG_n]=p_k N(n)$. Then, series expansion in the first term on the right hand side yields that the limiting distribution $p_k$ should satisfy:
\be\ba  &p_k(N(n+1)-N(n)) = -p_k N(n)A_k \hat q_n  + p_{k-1} N(n) \hat q_n A_{k-1}+ O(N(n)\hat q_n^2 (A_k^2+ A_{k-1}))\\
 &+ \sum_{\ell=2}^{(d-1)k/d} p_{k-\ell} N(n) \hat q_n  { A_{k-\ell} \choose \ell } \hat q_n^{\ell-1} (1+o(1))+ \ind_{\{k=d+1\}} (N(n+1)- N(n))
\ea\ee
Multiply both sides by $(N(n)\hat q_n)^{-1}$, and use that $N(n+1)-N(n)=\hat q_n A(n)$:
\be p_k \frac{A(n)}{N(n)} = -p_k A_k + p_{k-1} A_{k-1}  + \ind_{\{k=d+1\}}\frac{ A(n)}{N(n)}+ O( C_k \hat q_n) \ee
where $C_k$ contains the coefficients of all the smaller order terms.
Since $A(n)/N(n)\to d$ as $N(n)\to \infty$, and $\hat q_n= q_n(1+o(1)) \to 0$ as $n\to 0$, the limiting distribution $p_k$ should satisfy
\be\label{eq:pk-recursion} p_k (d+A_k) = p_{k-1} A_{k-1} + d \ind_{\{k=d+1\}}.\ee
Using the formula for $A_k$ in \eqref{eq::ean-active-k} we equivalently have
\be p_k = p_{k-1} \frac{(k-1)(d-1) - d^2 + d +2}{k(d-1) - d^2 +2d +2} + \ind_{\{ k=d+1\}}\frac{d}{2d+1}. \ee
The solution of this recursion is
\begin{equation} \label{k17}
    p_k=p_{d+1}\prod_{\ell=d+2}^k{\frac{\ell-1-d+\frac{2}{d-1}}{\ell-d + \frac{d+2}{d-1}}}=\frac{d}{2d+1}\frac{\Gamma(k-d+\frac{2}{d-1})}{\Gamma(1+\frac{2}{d-1})}\: \frac{\Gamma(2+\frac{d+2}{d-1})}{\Gamma(k+1-d +\frac{d+2}{d-1})},
\end{equation}
and hence by the properties of Gamma-function we obtain that
\[
    p_k\sim const\cdot k^{-\frac{2d-1}{d-1}},
\]
i.e.\ the `stationary solution' has a power-law decay with exponent in $(2,3]$ for $d\geq2$.
Now, the recursion \eqref{eq:conditional-recursion} initially is not stationary, hence we still need to show that
\begin{equation} \label{k3}
    \varepsilon_k(n):=\Ev[  N_k(n)|\CG_n]-  p_k N(n)
\end{equation}
is tending  to zero conditionally on $\CG_n$.
Using \eqref{eq:conditional-recursion} and \eqref{eq:pk-recursion} it is elementary to check that the following recursion holds for the error terms defined in \eqref{k3}:
\begin{equation} \begin{aligned} \label{eq:epsilon-recursion}
    \varepsilon_k&(n+1) =   \varepsilon_k(n)(1-\hat q_n)^{A_k}
  +   \varepsilon_{k-1}(n) A_{k-1}\hat q_n(1-\hat q_n)^{A_{k-1}-1}\\
  &-    p_k\left( N(n+1)-N(n)(1-\hat q_n)^{A_{k}}-(d+A_k) \hat q_n N(n) \right)\\
  &-     p_{k-1} N(n) \hat q_n A_{k-1}\left(1- (1-\hat q_n)^{A_{k-1}-1}\right) -\ind_{\{ k=d+1\} }( d N(n)-A(n))\hat q_n. \\
  &+ \sum_{\ell=2}^{\frac{d-1}{d}k}\Ev[  N_{k-\ell}(n)|\CG_n]\binom{A_{k-\ell}}{\ell}\hat q_n^{\ell}(1-\hat q_n)^{A_{k-\ell}-\ell}.
\end{aligned} \end{equation}
Let us denote by $\Delta_k(n):=|\varepsilon_k(n+1)-  \varepsilon_k(n)(1-\hat q_n)^{A_k}|$, i.e.\ the absolute value of the sum of the terms on the right hand side except the first term.
Since $(1-\hat q_n)^{A_k}<1$, we immediately get the upper bound
\[  \varepsilon_k(n+1) \le \varepsilon_k(0)+\sum_{i=1}^{n} \Delta_k(i). \]
With series expansion in the third and fourth term in \eqref{eq:epsilon-recursion} and the identity $dN(n)-A(n)=d+1$ yields the upper bound
\be\label{eq:delta1}\ba \frac{\Delta_k(n)}{N(n) \hat q_n }&\le \frac{\ve_{k-1}(n)}{N(n)} A_{k-1}+  p_k \left( \frac{d+1}{N(n)} + O(\hat q_n^2 A_k^2) \right) +\ind_{\{ k=d+1\} }\frac{d+1}{N(n)} \\
&+ p_{k-1} \hat q_n A_{k-1}^2 + \sum_{\ell=2}^{\frac{d-1}{d}k}\binom{A_{k-\ell}}{\ell}\hat q_n^{\ell-1} , \ea\ee
where we used that $\Ev[N_k(n)|\CG_n]\le N(n)$ holds for all $n,k$. Now clearly we have
 \[\Delta_{d+1}(i)=\hat q_i (d+1)(1+p_{d+1})+o(\hat q_i)\] for any $i\in \N$. Note that $\ve_{d+1}(0)\le d+2$ hence for some constants $C_{d+1}< C_{d+1}'<\infty$
 \[ \ve_{d+1}(n)\le d+2+C_{d+1} \sum_{i=1}^{n-1}\hat q_i(1+o(1))\le C_{d+1}'  \sum_{i=1}^{n-1}\hat q_i.\] Let us inductively assume that $\ve_{k-1}(i)\le  \ve_{k-1}(0)+C_{k-1}' \Big(\sum_{j=1}^{i-1}\hat q_j\Big)^{k-1-d}$ holds uniformly in $i$.
To carry out the inductive step for $k$ we can sum up \eqref{eq:delta1} and use the induction hypothesis to get
\be\label{k8}\ba
  \ve_k(n)&=\ve_k(0)+\sum_{i=1}^{n-1}\hat q_i \left( A_{k-1} C_{k-1}' \Big(\sum_{j=1}^{i-1} \hat q_j\Big)^{k-1-d} + c_{k,1} + c_{k,2} \hat q_i\right) \\
  &\le \ve_{k}(0) +c_{k,3} \sum_{i=1}^{n-1}\hat q_i + c_{k,4}\Big(\sum_{j\le i}^{n-1} \hat q_i\Big)^{k-d}\le \ve_{k}(0) + C_k'\Big(\sum_{i=1}^{n-1}\hat q_i \Big)^{k-d}, \ea
\ee
where in the second sum we can use that $\max \hat q_i \le 1$. Hence the induction step is satisfies, so we have just shown that $\ve_{k}(n)\le \ve_{k}(0) + C_k'\Big(\sum_{i=1}^{n-1}\hat q_i\Big)^{k-d}$ holds uniformly in $n$, which finishes the proof. \end{proof}

\bibliographystyle{abbrv}
\bibliography{bibliography}
\addcontentsline{toc}{section}{References}

\end{document}